\documentclass[10pt]{amsart}

\usepackage{amsfonts, amssymb, amsmath, enumerate}

\numberwithin{equation}{section}

\DeclareMathOperator{\E}{\mathbb{E}}

\def \N {\mathbb{N}}

\def \R {\mathbb{R}}

\def \E {\mathbb{E}}

\def \vol {{\rm vol}}

\def \etc {,\ldots,}

\newtheorem{theorem}{Theorem}[section]
\newtheorem{proposition}[theorem]{Proposition}
\newtheorem{corollary}[theorem]{Corollary}
\newtheorem{lemma}[theorem]{Lemma}

\theoremstyle{remark}

\begin{document}

\title{On maximal hyperplane sections of the unit ball of $l_p^n$ for $p>2$}

\author{Hermann K\"onig (Kiel)}


\begin{abstract}
The maximal hyperplane section of the $l_\infty^n$-ball, i.e. of the $n$-cube, is the one perpendicular to $\frac 1 {\sqrt 2} (1,1,0 \etc 0)$, as shown by Ball. Eskenazis, Nayar and Tkocz extended this result to the $l_p^n$-balls for very large $p \ge 10^{15}$. By Oleszkiewicz, Ball's result does not transfer to $l_p^n$ for $2 < p < p_0 \simeq 26.265$. Then the hyperplane section perpendicular to the main diagonal yields a counterexample for large dimensions $n$. Suppose that $p_0 \le p < \infty$. We show that the analogue of Ball's result holds in $l_p^n$-balls for all hyperplanes with normal unit vectors $a$, if all coordinates of $a$ have modulus $\le \frac 1 {\sqrt 2}$ and $p$ has distance $\ge 2^{-p}$ to the even integers. Under similar assumptions, we give a Gaussian upper bound for $20 < p < p_0$.
\end{abstract}

\maketitle

\begin{center}
\textit{Dedicated to the memory of Albrecht Pietsch}
\end{center}

\section{Introduction and main results}

It is not an easy task to determine extremal volumes of hyperplane sections or projections onto hyperplanes of convex bodies, even for specific convex sets. The proofs of known results are mainly based on probabilistic techniques and Fourier analysis, see Koldobsky \cite{K}, Chapters 3, 7 and 8.
Ball \cite{B} proved in a celebrated paper that the hyperplane section of the $n$-cube orthogonal to $a^{(2)}:=\frac 1 {\sqrt 2} (1,1,0 \etc 0) \in \R^n$ has maximal volume among all hyperplane sections. Hadwiger \cite{Ha} and Hensley \cite{He} had shown before that coordinate hyperplanes, e.g. those perpendicular to $(1,0 \etc 0) \in \R^n$, yield minimal volume central hyperplane cubic sections. \\

Meyer and Pajor \cite{MP} considered the analogue in $l_p^n$ for $1 \le p < \infty$. They proved that the normalized volume of central hyperplane sections of the unit ball $B_p^n$ of $l_p^n$ is monotone increasing in $p$, implying that coordinate hyperplanes provide the minimal sections for $2 < p < \infty$ and the maximal sections for $1 \le p < 2$. The minimal sections of $B_p^n$ for $1 \le p < 2$ are those orthogonal to the main diagonals, e.g. to $a^{(n)}:=\frac 1 {\sqrt n} (1, \etc 1) \in \R^n$, see \cite{MP} for $p=1$ and Koldobsky \cite{K} for $1 < p < 2$. \\

The question of the maximal hyperplane sections of $B_p^n$ for $2 < p < \infty$ is more complicated, since the answer may depend on $p$ as well as on the dimension $n$. Oleszkiewicz \cite{O} showed that Ball's result does not transfer to $B_p^n$ for $2 < p < p_0 \simeq 26.265$. In that case, the hyperplane section perpendicular to a main diagonal has larger volume than the one orthogonal to $a^{(2)}$ for large dimensions $n$. A quantitative estimate for which dimensions this happens is given in K\"onig \cite{K1}. On the other hand, for very large $p \ge 10^{15}$, Eskenazis, Nayar and Tkocz proved that Ball's result is stable for $l_p^n$: then $B_p^n \cap (a^{(2)})^\perp$ has maximal hyperplane section volume for all dimensions $n$. They call it "resilience of cubic sections". We extend this result to $p_0 \le p < \infty$ for those hyperplanes with normal unit vectors $a$, if all coordinates of $a$ have modulus $\le \frac 1 {\sqrt 2}$ and $p$ has distance $\ge 2^{-p}$ to the even integers. For $20 < p < p_0$ we give an asymptotic Gaussian type upper bound under similar assumptions. \\

Finding extremal projections of convex bodies onto hyperplanes is dual to determining extremal sections. For $l_p^n$-balls, the known results for sections essentially transfer to projections, if one interchanges $p$ and $q=p'$, $\frac 1 p + \frac 1 {p'} = 1$, and minimal and maximal volume results. For instance, Ball's result corresponds to the fact that the projection of the $l_1^n$-ball onto the hyperplane $(a^{(2)})^\perp$ has minimal volume, see Barthe, Naor \cite{BN}. The latter result is a consequence of Cauchy's projection formula and the optimal lower bound in the $L_1$-Khintchine inequality, which was found by Szarek \cite{S}. The Khintchine inequalities are important in Banach Space Theory and Operator Theory. They are useful, in particular, when studying absolutely $p$-summing operators, see Pietsch \cite{P}, Chapter 17. For $l_q^n$-balls, determining extremal hyperplane projections is equivalent to finding the best constants in a Khintchine-type inequality for random variables with a certain exponential density, see Barthe, Naor \cite{BN}. Identifying extremal hyperplane sections of $l_p^n$-balls also amounts to finding the best constants in a generalized Khintchine inequality, see the remarks after Proposition \ref{prop1}. Even though the results for projections are dual to those for sections, the known proofs are not obtained by duality. This is a consequence of the fact that volume does not behave well under duality.
Nayar and Tkocz \cite{NT} gave a beautiful survey on extremal sections and projections of classical convex bodies. \\

To formulate our results on sections of $l_p^n$ precisely, we introduce the following notations. For $1 \le p \le \infty$ and any natural positive number $n \in \N$, denote the closed unit ball of $l_p^n$ by
$$B_p^n := \{ \ x = (x_j)_{j=1}^n \in \R^n \ \big| \ \|x \|_p := ( \sum_{j=1}^n |x_j| )^{\frac 1 p} \le 1 \} \ . $$
Let $S^{n-1} := \{ \ x \in \R^n \ \big| \ \| x \|_2 =1 \ \}$, $a \in S^{n-1}$ be a direction vector and $a^\perp$ the hyperplane orthogonal to $a$. We define the normalized section function $A_{n,p}$ by
$$ A_{n,p}(a) := \frac{\vol_{n-1}(B_p^n \cap a^\perp)}{\vol_{n-1}(B_p^{n-1})} \ . $$
For $1 \le k \le n$, let $a^{(k)} := \frac 1 {\sqrt k} (\underbrace{1 \etc 1}_k , 0 \etc 0) \in S^{n-1} \subset \R^n$. Then Ball's theorem states $A_{n,\infty}(a) \le A_{n,\infty}(a^{(2)}) = \sqrt 2$ for all $n \in \N$ and $a \in S^{n-1}$. Eskenazis, Nayar and Tkocz's result reads $A_{n,p}(a) \le A_{n,p}(a^{(2)}) = 2^{\frac 1 2 - \frac 1 p}$ for all $n \in \N$, $a \in S^{n-1}$ and $p \ge 10^{15}$. Note that $A_{n,p}(a^{(2)})$ does not depend on $n$. \\

Oleszkiewicz \cite{O} showed that $\lim_{n \to \infty} A_{n,p}(a^{(n)}) > A_{n,p}(a^{(2)})$ holds for all $2 < p < p_0 \simeq 26.265$, where $p_0$ is the unique solution of $\frac 3 \pi \frac{2^{\frac 2 p} \Gamma \left(1+\frac 1 p \right)^3}{\Gamma \left(1+\frac 3 p \right)} = 1 $ for $ p \in (2,\infty)$.
More precisely, we have
$\lim_{n \to \infty} A_{n,p}(a^{(n)}) = \left(\frac 6 \pi \frac{\Gamma \left(1+ \frac 1 p \right)^3}{\Gamma \left(1+ \frac 3 p \right)} \right)^{\frac 1 2} > A_{n,p}(a^{(2)}) = 2^{\frac 1 2 - \frac 1 p}$, see \cite{O} and also \cite{K1}. \\

Let $a = (a_j)_{j=1}^n \in S^{n-1}$. In view of the symmetries of $B_p^n$, we may suppose that $a_1 \ge a_2 \ge \cdots \ge a_n \ge 0$. We assume this in the statement of our two main results:

\begin{theorem}\label{th1}
Let $p_0 \le p < \infty$ and $dist(p, 2 \N) \ge 2^{-p}$. Then for all $n \in \N$, $n \ge 2$ and $a \in S^{n-1}$
$$A_{n,p}(a) \le A_{n,p}(a^{(2)}) = 2^{\frac 1 2 - \frac 1 p} \ , $$
with the possible exception of those $a$ with $a_1 \in [\frac 1 {\sqrt 2}, \frac{2^{\frac 1 p}}{\sqrt 2}] =: I_p$, an interval of length $|I_p| < \frac 1 {2 p}$. If $a_1 \in I_p$ or if $dist(p, 2 \N) < 2^{-p}$, we have at least
$$A_{n,p}(a) < 2^{\frac 1 p} A_{n,p}(a^{(2)}) \ . $$
The restriction on $p$ can be weakened to $dist(p, 2 \N) \ge p^{-p}$ for $p \ge 50$ and to $dist(p, 2 \N) \ge p^{-5 p}$ for $p \ge 175$.
\end{theorem}

Most likely, the assumption that $p$ has a certain distance to the even integers is not necessary. By the result of \cite{ENT}, it is not needed, if $p \ge 10^{15}$.

\begin{theorem}\label{th2}
Let $20 < p < p_0$, $dist(p, 2 \N) \ge (\frac{14}{15})^p$. Then for all $n \in \N$, $n \ge 2$ and $a \in S^{n-1}$ with $a_1 \le \frac 1 {\sqrt 2}$
$$A_{n,p}(a) \le  \lim_{n \to \infty} A_{n,p}(a^{(n)}) = \left(\frac 6 \pi \frac{\Gamma (1+ \frac 1 p)^3}{\Gamma(1+ \frac 3 p)} \right)^{\frac 1 2} \ . $$
The result is also true if $a_1 \ge \frac 1 {\sqrt 2} + \frac 1 {3p} + \frac 1 {150}$.
\end{theorem}

The first part of \ref{th2} probably holds for all $2 < p < p_0$. The restriction $p > 20$ is required by estimates in the proof. \\

We sketch the main ideas of the proof of both theorems and describe the organization of the paper. As in Ball's cube slicing result, the proofs use Fourier transform representations of moments of independent sums, a technique also employed by Haagerup \cite {H} in his proof of the best constants in the Khintchine inequality. Ball's result is based on  P\'olya's \cite{Po1} formula $A_{n,\infty}(a) = \frac 2 \pi \int_0^\infty \prod_{j=1}^n \gamma_\infty(a_j s) \ ds $, $a=(a_j)_{j=1}^n \in S^{n-1}$, with $\gamma_\infty(s) = \frac{\sin(s)} s$, $s \ne 0$. The main ingredient of his proof is the integral inequality $h_\infty(u) \le h_\infty(2)$ for all $u \ge 2$, where $h_\infty(u) := \sqrt u \int_0^\infty |\gamma_\infty(s)|^u \ ds $, which together with H\"older's inequality allows to prove \\ $A_{n,\infty}(a) \le A_{n,\infty}(a^{(2)})$ for all vectors $a$ with $|a_j| \le \frac 1 {\sqrt 2}$, $j=1, \cdots , n$. There is an analogue of P\'olya's formula for $A_{n,p}(a)$ in terms of the normalized Fourier transform $\gamma_p(s)$ of $\exp(-|r|^p)$, see Proposition \ref{prop1}. Let $h_p(u) := \sqrt u \int_0^\infty |\gamma_p(s)|^u \ ds$. Our main technical result is Theorem \ref{th3}, namely that $h_p(u) \le h_p(2)$ for all $u \ge 2$, provided that $p \ge p_0$ and that $p$ is not too close to an even integer. Together with H\"older's inequality and Proposition \ref{prop1}, this yields a proof of Theorem \ref{th1} for hyperplanes with normal unit vectors which have only coordinates of modulus $\le \frac 1 {\sqrt 2}$, see Section 2. \\
Nazarov and Podkorytov \cite{NP} found an ingenious proof of Ball's integral inequality, using the distribution functions of $f(s)=|\gamma_\infty(s)|$ and $g(s) = \exp(-\frac{s^2}{2 \pi})$, cf. Proposition \ref{prop2}. We apply their method to prove the integral inequality in Theorem \ref{th3}. Consider $f(s) = |\gamma_p(s)|$ and $g(s) = \exp(-d_p s^2)$, where $d_p>0$ is such that $||f||_{L_2(0,\infty)} = ||g||_{L_2(0,\infty)}$. Nazarov and Podkorytov's method yields inequalities of the form $\int_0^\infty |f(s)|^u \ ds \le \int_0^\infty |g(s)|^u \ ds$ for all $u \ge 2$, provided that the difference of the distribution functions $F-G$ of $f-g$ changes sign just once in the right direction. Then
\begin{align*}
h_p(u) & = \sqrt u \int_0^\infty |\gamma_p(s)|^u \ ds \le \sqrt{u} \int_0^\infty \exp(-d_p s^2)^u \ ds \\
& = \sqrt{2} \int_0^\infty \exp(-d_p s^2)^2 \ ds = h_p(2)
\end{align*}
by the definition of $d_p$. To have $F(x)-G(x) < 0$ for values $x$ close to $x < 1$, one needs that $f(s) = |\gamma_p(s)| \le \exp(-d_p s^2) = g(s)$ for $s$ near zero. We show that $|\gamma_p(s)| \le \exp(-c_p s^2)$ holds for $0 \le s \le 3$, see Proposition \ref{prop3}, where $c_p > 0$ is the optimal constant for this inequality close to $s=0$. To have $f(s) \le g(s)$ near $0$, one needs that $d_p \le c_p$ holds. This condition is just Oleszkiewicz' restriction on $p$, namely $p \ge p_0$, see Lemma \ref{lem1}. While the distribution function of $g$ is simple, $G(x) = \sqrt{\frac 1 {d_p} \ln(\frac 1 x)}$, the distribution function of $f = |\gamma_p|$ is rather complicated. However, its determination is independent of the specific index $p_0$. While the functions $\gamma_p$ are positive for $1 \le p \le 2$, see Koldobsky \cite{K}, for $p>2$, $p \notin 2 \N$, they have finitely many real zeros, and for large arguments $s$ show a negative power type decay, see P\'olya \cite{Po}. We prove an error estimate for this decay for large $s$, see Proposition \ref{prop4}. For $p>2$, $p \in 2 \N$, the functions $\gamma_p$ have infinitely many real zeros and decay exponentially fast at infinity, see Boyd \cite{Bo} for an asymptotic expansion, but without error estimates. For relatively small $s>0$, the functions $\gamma_p$ show a damped $\frac{\sin(s)} s$-type behavior, with faster decay for smaller $p$. To find estimates the distribution function $F$ of $|\gamma_p|$, we approximate $\exp(-|r|^p)$ by linear splines and calculate their Fourier transform, which are roughly of $\frac{\sin(s)} s$-type decay for $0 \le s \le c p$, see Lemmas \ref{lem3} and \ref{lem4}. Their distribution functions are determined in Proposition \ref{prop5} and in the proof of Theorem \ref{th3}. For large $p$, say $p \ge 400$, the simple approximation of $\gamma_p$ in Lemma \ref{lem3} suffices, but for $p_0 \le p \le 400$, we need the better approximation of $\gamma_p$ in Lemma \ref{lem4}. Lemmas \ref{lem6} and \ref{lem7} are required to estimate the size of the first bumps of $|\gamma_p|$ and the distribution function $F(x)$ for $\frac 1 {20} \le x \le \frac 1 5$. \\
The difficulty determining $F$ is two-fold: first, it depends on $p$, and the decay of $|\gamma_p|$ for $0 \le s \le p$ is faster for smaller $p$, which requires a better approximation of $\gamma_p$, and secondly, for $p$ close to an even integer, the decay for large $s$ is difficult to handle: for $p \notin 2 \N$, the distribution function $F$ increases of inverse power type $\beta(p) x^{-\frac 1 {p+1}}$ for small $x > 0$ with $\lim_{dist(p, 2 \N) \to 0} \beta(p) =0$, while for $p \in 2 \N$ it increases just logarithmically, $(\ln (\frac 1 x))^{1-\frac 1 p}$. The last fact lead to the requirement in both theorems that $p$ should not be too close to the even integers. Our arguments are analytical, but for relatively small $p$ we rely on the numerical evaluation of some integrals. \\
In the case $a_1 > \frac 1 {\sqrt 2}$ Ball projects the hyperplane onto $(a^{(1)})^\perp$ and estimates the distortion factor. This method works in $l_p^n$-spaces only for $a_1 \ge \frac{2^{\frac 1 p}}{\sqrt 2}$, leaving the gap between $\frac 1 {\sqrt 2}$ and $\frac{2^{\frac 1 p}}{\sqrt 2}$. To close this gap, one would need a new argument, possibly of geometric nature. The technique used by Eskenazis, Nayar and Tkocz \cite{ENT} to cover the case $a_1 > \frac 1 {\sqrt 2}$ only works for very large $p$. \\

\section{Formulas and proofs}

Ball \cite{B} used  P\'olya's \cite {Po1} result
\begin{equation}\label{eq2.1}
A_{n,\infty}(a) = \frac 2 \pi \int_0^{\infty} \prod_{j=1}^n \frac{\sin(a_j s)}{a_j s} \ ds \ , \ a=(a_j)_{j=1}^n \in S^{n-1}
\end{equation}
to prove his theorem for $l_\infty^n$. There is a corresponding formula for $l_p^n$:

\begin{proposition}\label{prop1}
Let $1 \le p < \infty$, $n \in \N$ and $a = (a_j)_{j=1}^n \in S^{n-1}$, $a_j \ge 0$. Then
\begin{equation}\label{eq2.2}
A_{n,p}(a) = \Gamma \left(1+\frac 1 p \right) \frac 2 \pi \int_0^{\infty} \prod_{j=1}^n \gamma_p(a_j s) \ ds \ ,
\end{equation}
where
$$\gamma_p(s) := \frac 1 {\Gamma \left(1+ \frac 1 p \right)} \int_0^\infty \cos(s r) \ \exp(-r^p) \ dr \ , \ \gamma_p(0)=1 \ . $$
\end{proposition}

A proof of \eqref{eq2.2} by Fourier transform techniques can be found in Koldobsky \cite{K}, Theorem 3.2, using the normalization with
$\vol_{n-1}(B_p^{n-1}) = \frac{(2 \Gamma(1+ \frac 1 p))^{n-1}}{\Gamma(1 + \frac {n-1} p)}$. A different proof is given in K\"onig \cite{K1}, as a consequence of the following formula of Eskenazis, Nayar and Tkocz \cite{ENT}, used in the proof of their resilience of cube slicing result:
\begin{equation}\label{eq2.3}
A_{n,p}(a) = \Gamma \left(1 + \frac 1 p \right) \ \E_{(\xi,R)} \ \| \sum_{j=1}^n a_j R_j \xi_j \|_2^{-1} \ ;
\end{equation}
here the $\xi_j$ are i.i.d. random vectors uniformly distributed on the unit sphere $S^2 \subset \R^3$ and the $R_j$ are i.i.d. positive random variables with density \\ $\frac p {\Gamma(1 + \frac 1 p)} x^p \exp(-x^p) 1_{[x>0)}(x)$, independent of the $\xi_j$'s. Theorem \ref{th1} amounts to determining the optimal constant in a Khintchine type inequality for the $L_{-1}$-norm with respect to these (combined) variables.\\

For $1 \le p < 2$, the $\gamma_p$ are just the (positive) densities of the $p$-stable random variables. However, we need the formula for $2 < p < \infty$. By P\'olya \cite{Po}, for $2 < p < \infty$ and $p \notin 2 \N$, the $\gamma_p$ have real zeros, but only finitely many and infinitely many complex zeros. For $4 \le p \in 2 \N$, the $\gamma_p$ have infinitely many real zeros and no complex zeros, see \cite{Po} or Boyd \cite{Bo}. The order of decay of $\gamma_p$ is very different in the cases $p \notin 2 \N$ and $p \in 2 \N$, in the first case of negative power type, in the second of exponential decay. For $p \to \infty$, $\exp(-x^p) \to 1_{[0,1)}(x)$ pointwise. \\

By the dominated convergence theorem, $\gamma_p(s) \to \int_0^1 \cos(sr) dr = \frac{\sin(s)} s$ for $p \to \infty$. Ball used \eqref{eq2.1} and his integral inequality $h_\infty(u) \le h_\infty(2)$ for all $u \ge 2$, where $h_\infty(u) := \sqrt u \int_0^\infty |\frac{\sin(s)} s |^u ds$. Our proofs of Theorems \ref{th1} and \ref{th2} are based on \eqref{eq2.2} and the following analogue of Ball's integral inequality for $l_p^n$, which is the main technical result of this paper.

\begin{theorem}\label{th3}
Let $2 < p < \infty$, $c_p := \frac 1 6 \frac{\Gamma(1+ \frac 3 p)}{\Gamma(1 + \frac 1 p)}$, $d_p := \frac{(2^{\frac 1 p} \Gamma(1+ \frac 1 p))^2}{2 \pi}$. Put $$h_p(u) := \sqrt u \int_0^\infty |\gamma_p(s)|^u \ ds \ , \ 2 \le u < \infty \ . $$
Then
$$h_p(2) = \frac 1 2 \sqrt{\frac \pi {d_p}} \text{ \  and  \ } h_p(\infty) := \lim_{u \to \infty} h_p(u) = \frac 1 2 \sqrt{\frac \pi {c_p}} \ . $$
a) If $d_p \le c_p$, we have  $h_p(u) \le h_p(2)$ for all $u \ge 2$, assuming additionally $dist(p, 2 \N) \ge \frac 1 {A^p}$, where $A=2$. It also holds for $A=p$, if $p \ge 50$ and for $A=p^5$, if $p \ge 175$. \\
b) If $c_p < d_p$ and $p \ge 20$, we have $h_p(u) \le h_p(\infty)$ for all $u \ge 2$, assuming additionally $dist(p, 2 \N) \ge (\frac{14}{15})^p$.
\end{theorem}

We need some facts on the $\Gamma$-function. Since $\Gamma$ is log-convex, the Digamma function $\Psi := (\ln \Gamma)'$ is strictly increasing, $\Psi' > 0$. By Abramowitz, Stegun \cite{AS}, 6.3.16. and 6.4.10 we have for all $x > 0$
\begin{equation}\label{eq2.4}
\Psi(1+x) = -\gamma + \sum_{k=1}^\infty \frac x {k(x+k)} \ , \ \Psi'(1+x) = \sum_{k=1}^\infty \frac 1 {(k+x)^2} \ ,
\end{equation}
where $\gamma \simeq 0.5772$ is the Euler-Mascheroni constant.
In particular, $\Psi'(1) = \frac{\pi^2} 6$, $\Psi'(2) = \frac{\pi^2} 6 -1$ and $\Psi'(\frac 5 2) = \frac{\pi^2} 2 - \frac{40} 9$, \cite{AS}, 6.4.5. \\

\begin{lemma}\label{lem1}
The equation $c_p = d_p$ has exactly one solution $p_0 \in (2,\infty)$, $p_0 \simeq 26.265$, with $h_{p_0}(2) = h_{p_0}(\infty)$. For $p_0 < p < \infty$, $d_p < c_p$ and $h_p(\infty) < h_p(2)$. For $2 < p < p_0$, $c_p < d_p$ and $h_p(2) < h_p(\infty)$.
\end{lemma}

\begin{proof}
The equation $c_p = d_p$ is equivalent to Oleszkiewicz's condition in \cite{O}
$$\phi(p) := \frac 3 \pi \frac{2^{\frac 2 p} \Gamma \left(1 + \frac 1 p \right)^3}{\Gamma \left(1+ \frac 3 p \right)} = 1 \ . $$
Note that $\phi(2)=1$, $\lim_{p\to \infty} \phi(p) = \frac 3 \pi < 1$. We show that $\phi$ is strictly increasing in $p \in (2,p_1)$ and strictly decreasing in $p \in (p_1,\infty)$, where $p_1 \simeq 4.192$. We have
$$(\ln \phi)'(p) = - \frac 1 {p^2} \left(2 \ln 2 + 3 \Psi \left(1+ \frac 1 p \right) - 3 \Psi \left(1 + \frac 3 p \right) \right) \ , $$
We claim that $F(p) := \Psi(1+\frac 1 p)-\Psi(1+ \frac 3 p)$ is increasing for $p \in (2,\infty)$: Since $F'(p) = \frac 1 {p^2} (3 \Psi'(1+ \frac 3 p) - \Psi'(1+\frac 1 p))$, \eqref{eq2.4} yields $F'(p) = \frac 1 {p^2} \sum_{k=1}^\infty \frac{2 k^2 - \frac 6 {p^2}}{(k+\frac 3 p)^2 (k+\frac 1 p)^2} > 0$ for all $p \ge \sqrt 3$, in particular for $p \ge 2$. Thus $F$ and $2 \ln 2 + 3 F$ are strictly increasing in $(2,\infty)$, with $2 \ln 2 + 3 F(2) = -2(1- \ln 2) < 0$, using $\Psi(\frac 3 2) - \Psi(\frac 5 2) = - \frac 2 3$, cf. \cite{AS}, 6.3.4. However, for $p=5$, $2 \ln 2 + 3 F(5) > 0$. Hence there is exactly one zero $p_1$ of $2 \ln 2 + 3 F(p) = 0$ in the interval $(2,5)$, $p_1 \simeq 4.192$. For $2 < p < p_1$,
$(\ln \phi)' >0$, $\phi$ is increasing, and for $p_1 < p < \infty$, $(\ln \phi)' < 0$, $\phi$ is decreasing. For $p \to \infty$, $(\ln \phi)' \to 0$ and $\phi(p) \to \frac 3 \pi < 1$. Thus $\phi(p)=1$ has exactly one solution $p_0 \in (2, \infty)$, $p_0 \simeq 26.265$. One may check e.g. that $\phi(26) > 1 > \phi(27)$.
\end{proof}

\begin{lemma}\label{lem2}
(a) As a function of $p$, $d_p = \frac{(2^{\frac 1 p} \Gamma(1+ \frac 1 p))^2}{2 \pi}$ decreases in $[2,\infty)$, with $d_2 = \frac 1 4$ and $\lim_{p \to \infty} d_p = \frac 1 {2 \pi} \simeq 0.15915$. \\
(b) As a function of $p$, $c_p =  \frac 1 6 \frac{\Gamma(1+ \frac 3 p)}{\Gamma(1 + \frac 1 p)}$ decreases in $[2,p_2)$ and increases in $(p_2,\infty)$, where $p_2 \simeq 9.1147$, with $c_2= \frac 1 4$, $\lim_{p \to \infty} c_p = \frac 1 6$ and $c_{p_2} \simeq 0.15715$. \\
For $p=p_0 \simeq 26.265$ we have $c_{p_0} = d_{p_0} \simeq 0.1609$. Further $c_{15} \simeq 0.1584$.
\end{lemma}

\begin{proof}
(a) We have $\frac d {dp} (\ln d_p) = - \frac 2 {p^2} (\ln 2 + \Psi(1 + \frac 1 p)) < 0$. \\
(b) We get $\frac d {dp} (\ln c_p) = - \frac 1 {p^2} (\Psi(1+ \frac 1 p) - 3 \Psi(1+ \frac 3 p)) =: \frac{F(p)}{p^2}$. By \eqref{eq2.4} $\Psi'$ is decreasing and for $0 \le x \le \frac 3 2$, $\frac{\pi^2} 2 - \frac{40} 9 = \Psi'(\frac 5 2) \le \Psi'(1+x) \le \Psi'(1) = \frac{\pi^2} 6$. This implies for all $p \ge 2$ that
$$F'(p) = \frac 1 {p^2} \left(9 \Psi' \left(1+ \frac 3 p \right) - \Psi' \left(1+ \frac 1 p \right) \right) \ge \frac 1 {p^2} \left(\frac {13} 3 \pi^2 - 40 \right) > \frac 2 {p^2} >0 \ . $$
Hence $F$ is increasing in $[2,\infty)$. Since $F(9) < -0.012$, $F(10) > 0.084$, $F$ has exactly one zero $p_2 \in [2,\infty)$, $p_2 \simeq 9.1147$. Thus $c_p$ is decreasing in $[2,p_2)$ and increasing in $(p_2,\infty)$. For all $p \in [2,\infty)$, $c_p \ge c_{p_2} \simeq 0.15715$.
\end{proof}

We now prove Theorems \ref{th1} and \ref{th2} using \eqref{eq2.2} and Theorem \ref{th3} which we will verify later. \\

{\bf Proof of Theorem \ref{th1}} \\

i) Since $\gamma_p$ is the normalized Fourier cosine transform of $\exp(-r^p)$, Plancherel's theorem yields
\begin{align}\label{eq2.5}
h_p(2) & = \sqrt 2 \int_0^\infty \gamma_p(s)^2 \ ds = \frac{\sqrt 2}{\Gamma \left(1+\frac 1 p \right)^2} \frac{\pi} 2 \int_0^\infty \exp(-r^p)^2 \ dr \nonumber \\
& = \frac 1 {\Gamma \left(1+ \frac 1 p \right)^2} \frac\pi {\sqrt 2} \frac 1 {p 2^{\frac 1 p}} \int_0^\infty v^{\frac 1 p -1} \ \exp(-v) \ dv = \frac \pi {2^{\frac 1 2 + \frac 1 p} \Gamma \left(1 + \frac 1 p \right) } = \frac 1 2 \sqrt{\frac \pi {d_p}} \ .
\end{align}
For $u \ge 2$, $h_p(u) = \int_0^\infty \left|\gamma_p \left(\frac s {\sqrt u} \right) \right|^u \ ds$. Since $\cos(sr) = 1 - \frac{s^2r^2} 2 + O(r^4)$ near $r=0$  and
\begin{align*}
\frac 1 {\Gamma \left(1 + \frac 1 p \right)} \int_0^\infty \left(1-\frac{s^2}{2 u} r^2 \right) \ \exp(-r^p) \ dr & = 1 - \frac 1 {2 p} \frac{\Gamma \left(\frac 3 p \right)}{\Gamma \left(1+\frac 1 p \right)} \frac {s^2} u \\
& = 1 - \frac 1 6 \frac{\Gamma \left(1+\frac 3 p \right)}{\Gamma \left(1+\frac 1 p \right)} \frac{s^2} u = 1 - c_p \frac{s^2} u \ ,
\end{align*}
the dominated convergence theorem yields
\begin{align}\label{eq2.6}
h_p(\infty) & := \lim_{u \to \infty} h_p(u) = \lim_{u \to \infty} \int_0^\infty \left(1-c_p \frac {s^2} u \right)^u \ ds \nonumber \\
& = \int_0^\infty \exp(-c_p s^2) \ ds = \frac 1 2 \sqrt{\frac \pi {c_p}} = \sqrt{\frac {3 \pi} 2 \frac{\Gamma(1+\frac 1 p)}{\Gamma(1+\frac 3 p)}} \ .
\end{align}

ii) Let $a=(a_j)_{j=1}^n \in S^{n-1}$. Assume first that all $a_j$ satisfy $0 \le a_j \le \frac 1 {\sqrt 2}$. We may assume $a_j>0$ for all $j$. In Theorem \ref{th1} $d_p \le c_p$ and hence by Theorem \ref{th3} $h_p(u) \le h_p(2)$ for all $u \ge 2$. Then \eqref{eq2.2}, H\"older's inequality with $\sum_{j=1}^n a_j^2 =1$ and $h_p(u) \le h_p(2)$ for $u = a_j^{-2} \ge 2$ imply
\begin{align*}
A_{n,p}(a) & = \Gamma \left(1+\frac 1 p \right) \frac 2 \pi \int_0^\infty \prod_{j=1}^n \gamma_p(a_j s) \ ds \\
& \le  \Gamma \left(1+\frac 1 p \right) \frac 2 \pi \prod_{j=1}^n \left(\int_0^\infty |\gamma_p(a_j s)|^{a_j^{-2}} \ ds \right)^{a_j^2} \\
& = \Gamma \left(1+\frac 1 p \right) \frac 2 \pi \prod_{j=1}^n \left(\frac 1 {a_j} \int_0^\infty |\gamma_p(s)|^{a_j^{-2}} \ ds \right)^{a_j^2} \\
& = \Gamma \left(1+\frac 1 p \right) \frac 2 \pi \prod_{j=1}^n h_p(a_j^{-2})^{a_j^2} \\
& \le \Gamma \left(1+\frac 1 p \right) \frac 2 \pi h_p(2) = 2^{\frac 1 2 - \frac 1 p} = A_{n,p}(a^{(2)}) \ ,
\end{align*}
where the last two equalities are a direct consequence of the formula for $h_p(2)$ in \eqref{eq2.5} and of \eqref{eq2.2}. \\

iii) Let $a=(a_j)_{j=1}^n \in S^{n-1}$, $a_j \ge 0$ and $a_1 > \frac 1 {\sqrt 2}$. Then $a_j < \frac 1 {\sqrt 2}$ for all $j \ge 2$. Consider the "cylinder"
$$Z := \{ x = (x_1 \etc x_n) \in \R^n \ | \ (x_2 \etc x_n) \in B_p^{n-1} \ , \ x_1 \in\R \ \} \ . $$
Then $B_p^n \cap a^\perp \subset Z \cap a^\perp$. Let $Q : \R^n \to \R^n$ be the orthogonal projection onto $(a^{(1)})^{\perp}$ and $T = \R^n \to \R^n$ be given by $T(t a + x) = t a^{(1)} + Q x$, where $x \in a^\perp$. In matrix terms
$$T = \left( \begin{array}{rrrr}
a_1 & a_2 & ...  & a_n \\
-a_2a_1 & 1-a_2^2 & ... & -a_2a_n \\
... & ... & ... & ...  \\
-a_na_1 & ... & -a_na_{n-1} & 1-a_n^2
\end{array} \right) \ . $$
Then $Ta = a^{(1)}$ and for any $x =(x_1, x_2 \etc x_n)  \in  \R^n$ with $\langle a,x \rangle = 0$ we have $Tx = (0,x_2 \etc x_n)$. Thus $T$ projects $Z \cap a^\perp$ onto $B_p^{n-1} = B_p^n \cap (a^{(1)})^\perp$. Calculation shows that $\det (T) = a_1$ and hence
$$A_{n,p}(a) \le \frac{\vol_{n-1}(Z \cap a^\perp)}{\vol_{n-1}(B_p^{n-1})} = \frac 1 {a_1} \le 2^{\frac 1 2 - \frac 1 p} = A_{n,p}(a^{(2)}) \ , $$
if $a_1 \ge 2^{\frac 1 p - \frac 1 2}$. This proves $A_{n,p}(a) \le A_{n,p}(a^{(2)})$ for all $a \in S^{n-1}$ except possibly for those $a$ with
$a_1 \in (\frac 1 {\sqrt 2},\frac {2^{\frac 1 p}}{\sqrt 2}) = I_p$, an interval of length $|I_p| < \frac 1 {2p}$ for $p \ge p_0$. \\

iv) Suppose that $a_1 \in I_p$ or that $dist(p, 2 \N) < 2^{-p}$. If $a_1 \in I_p$, choose $2 < p < q < \infty$ with
$\frac 1 {\sqrt 2} < \frac{2^{\frac 1 q}}{\sqrt 2} < a_1$ and $dist(q, 2 \N) \ge 2^{-q}$. If $a_1 \notin I_p$, but $dist(p, 2 \N) < 2^{-p}$, choose $2 < p < q < \infty$ with $dist(q, 2 \N) \ge 2^{-q}$. Since $A_{n,p}(a)$ is increasing in $p$ by Meyer, Pajor \cite{MP}, we have $A_{n,p}(a) \le A_{n,q}(a)$. Then by iii)
$$A_{n,p}(a) \le A_{n,q}(a) \le A_{n,q}(a^{(2)}) = 2^{\frac 1 p - \frac 1 q} A_{n,p}(a^{(2)}) < 2^{\frac 1 p} A_{n,p}(a^{(2)}) \ . $$     \hfill $\Box$

{\bf Proof of Theorem \ref{th2}} \\

i) Let $a=(a_j)_{j=1}^n \in S^{n-1}$ with $0 \le a_j \le \frac 1 {\sqrt 2}$ for all $j=1 \etc n$. We may assume $a_j >0$ for all $j$. Suppose that $20 < p < p_0$. Then $c_p < d_p$ and $h_p(u) \le h_p(\infty)$ for all $u \ge 2$ by Theorem \ref{th3}. Similarly as in part ii) of the previous proof we find with $u = a_j^{-2} \ge 2$
\begin{align*}
A_{n,p}(a) & \le \Gamma \left(1+\frac 1 p \right) \frac 2 \pi \left(\prod_{j=1}^n h_p(a_j^{-2}) \right)^{a_j^2} \\
& \le \Gamma \left(1+\frac 1 p \right) \frac 2 \pi h_p(\infty) = \sqrt{\frac 6 \pi \frac{\Gamma \left(1+\frac 1 p \right)^3}{\Gamma \left(1 + \frac 3 p \right)}} = \lim_{n \to \infty} A_{n,p}(a^{(n)}) \ .
\end{align*}
The last equality is a consequence of \eqref{eq2.2} and the dominated convergence theorem,
\begin{align*}
\lim_{n \to \infty} A_{n,p}(a^{(n)}) & = \lim_{n \to \infty} \Gamma \left(1 + \frac 1 p \right) \frac 2 \pi \int_0^\infty \gamma_p \left(\frac s {\sqrt n} \right)^n \ ds \\
& = \lim_{n \to \infty} \Gamma \left(1 + \frac 1 p \right) \frac 2 \pi \int_0^\infty \left(1-c_p \frac {s^2} n \right)^n \ ds \\
& = \Gamma \left(1 + \frac 1 p \right) \frac 2 \pi \int_0^\infty \exp(-c_p s^2) \ ds \\
& = \Gamma \left(1 + \frac 1 p \right) \frac 2 \pi \frac 1 2 \sqrt{\frac\pi {c_p}} = \sqrt{\frac 6 \pi \frac{\Gamma \left(1+\frac 1 p \right)^3}{\Gamma \left(1 + \frac 3 p \right)}} \ .
\end{align*}

ii) If $a_1 > \frac 1 {\sqrt 2}$, the same argument as in iii) of the previous proof shows that $A_{n,p}(a) \le \frac 1 {a_1}$. Hence Theorem \ref{th2} will be valid for all $a \in S^{n-1}$ with $a_1 \ge \sqrt{\frac \pi 6 \frac{\Gamma(1+\frac 3 p)}{\Gamma(1 + \frac 1 p)^3}} =: \phi(p)$. The function $\phi$ is convex in $(2,p_0)$: Calculation yields that $\textup{signum}(\phi''(p)) = \textup{signum}(\psi(p))$, where
\begin{align*}
\psi(p) & = \left[\Psi \left(1+ \frac 3 p \right)- \Psi \left(1+\frac 1 p \right) \right] \left[3 \Psi \left(1+ \frac 3 p \right)- 3 \Psi \left(1+\frac 1 p \right) + 4 p \right] \\
& + \left[6 \Psi' \left(1+ \frac 3 p \right) - 2 \Psi' \left(1+ \frac 1 p \right) \right] \ .
\end{align*}
The first summand is positive, since $\Psi$ is increasing due to the log-convexity of $\Gamma$, and for the second summand \eqref{eq2.4} yields
$$ 6 \Psi' \left(1+ \frac 3 p \right) - 2 \Psi' \left(1+ \frac 1 p \right) = \sum_{k=1}^\infty \frac 4 {p^2} \frac {k^2 - \frac 3 {p^2}}{(k+\frac 3 p)^2(k+\frac 1 p)^2} > 0 \ . $$
We have that $\phi(p) > \frac 1 {\sqrt 2} + \frac 1 {3p} + \frac 1 {150}$ holds for $p=20$ and $p=p_0$. Thus by convexity of $\phi$, $a_1 > \frac 1 {\sqrt 2} + \frac 1 {3p} + \frac 1 {150}$ implies $a_1 > \phi(p)$ for all $20 \le p \le p_0$.    \hfill $\Box$   \\

{\bf Conjecture.}  We conjecture that for $2 < p < p_0$ and $n \ge 2$ the maximal hyperplane section of $B_p^n$ is either given by $(a^{(2)})^\perp$ or by $(a^{(n)})^\perp$,
$$A_{n,p}(a) \le \max ( A_{n,p}(a^{(2)}) , A_{n,p}(a^{(n)}) ) \ , $$
where $a^{(2)}$ only occurs for relatively small $n \le N(p)$. The motivation for the conjecture is that the $p$-analogue $h_p(u)$ of Ball's integral function, which by the above proof is closely related to $A_{n,p}(a^{(n)})$ for $u=n$, seems to decrease first for $u > 2$ close to $2$ and then increase to $\lim_{n \to \infty} h_p(n) > h_p(2)$, at least for $5 \le p < p_0$. The value $N(p)$ should be approximately the smallest value $n$ such that $h_p(n) > h_p(2)$. Numerical evaluation of \eqref{eq2.2} yields e.g.
$$A_{3,6}(a^{(2)}) = 2^{\frac 1 3} \simeq 1.260 > A_{3,6}(a^{(3)}) \simeq 1.250 \ , $$
$$A_{4,8}(a^{(2)}) = 2^{\frac 3 8} \simeq 1.297 > A_{4,8}(a^{(4)}) \simeq 1.295 >  A_{4,8}(a^{(3)}) \simeq 1.270 \ . $$
For $2 < p < 4.3$, $h_p$ seems to increase in the full interval $(2,\infty)$, at least we have by numerical evaluation that
\begin{align*}
h_p'(2) & = \frac 1 {2 \sqrt 2} \int_0^\infty \gamma_p(s)^2 \ ds + \sqrt 2 \int_0^\infty \gamma_p(u)^2 \ \ln|\gamma_p(u)| \ du  \\
& = \frac 1 8 \sqrt{\frac \pi {d_p}} + \sqrt 2 \int_0^\infty \gamma_p(u)^2 \ \ln|\gamma_p(u)| \ du > 0
\end{align*}
e.g. $h_4'(2) > 0.009$, $h_{4.3}'(2) > 0.002$. But $h_{4.5}'(2) < -0.002$. Thus we conjecture that for $2 < p < 4.3$, the main diagonal $a^{(n)}$ should yield the maximal hyperplane section for all dimensions $n$. In this respect, note also that, by Koldobsky \cite{K}, Chapter 7, for $p < 2$ we have $A_{n,p}(a^{(n)}) \le A_{n,p}(a)$ for all $a \in S^{n-1}$, which implies $\frac d {dp} A_{n,p}(a)|_{p=2} \le \frac d {dp} A_{n,p}(a^{(n)})|_{p=2}$, since $A_{n,2}(a)=A_{n,2}(a^{(n)})$. If the inequality were strict, we could conclude that $A_{n,p}(a) \le A_{n,p}(a^{(n)})$ for $p>2$ close to $2$. The conjecture is that this holds up to $p=4.3$ for all $a \in S^{n-1}$. \\

To prove Theorem \ref{th3}, we will use the following ingenious distribution function result of Nazarov, Podkorytov \cite{NP}. Let $(\Omega,\mu)$ be a measure space. The distribution function $F$ of a measurable function $f:\Omega \to \R_{\ge 0}$ is defined by
$$F(x) = \mu \{s \in \Omega \ | \ f(s) > x \} \ ,  \ x>0 \ . $$

\begin{proposition}\label{prop2}
Let $(\Omega,\mu)$ be a measure space, $f, g : \Omega \to \R_{\ge 0}$ be measurable and $F$, $G$ be their distribution functions, assumed to be finite for every $x>0$. Suppose that there is $u_0 \in (0,\infty)$ such that $g^u - f^u \in L_1(\Omega,\mu)$ for all $u \ge u_0$. Assume further that there is $x_0 > 0$ such that the difference $G - F$ changes sign from $-$ to $+$, i.e. $G(x) \le F(x)$ for all $x \in (0,x_0)$ and $G(x) \ge F(x)$ for all $x> x_0$. Then the function
$$h(u) := \frac 1 {u x_0^u} \int_{\Omega} (g^u - f^u ) \ d \mu $$
is increasing in $u \in [u_0,\infty)$. In particular, if $\int_\Omega (g^{u_0} - f^{u_0}) \ d \mu \ge 0$ holds, we also have $\int_\Omega (g^u - f^u) \ d \mu \ge 0$ for all $u \ge u_0$.
\end{proposition}

We will apply this to $f = |\gamma_p|$ and $g(s) = \exp(-d_p s^2)$ and $u_0 = 2$ to prove Theorem \ref{th3} (a). Note that
$$\sqrt u \int_0^\infty \exp(-d_p s^2)^u \ ds = \frac 1 2 \sqrt{\frac \pi {d_p}}$$
is independent of $u$ and we get a bound of $h_p(u) = \sqrt u \int_0^\infty |\gamma_p(s)|^u \ ds$ for all $u \ge u_0$. \\

While the distribution function of $G$ is easily determined, $G(x) = \sqrt{\frac 1 {d_p} \ln (\frac 1 x)}$, the estimate of the distribution function of $|\gamma_p|$ is more complicated and requires some information on the decay of $\gamma_p(s)$, which we will investigate in the next section. \\

\section{Decay of $\gamma_p$}

To estimate the distribution function of $|\gamma_p|$, we need some results on the decay of $\gamma_p(s)$, for $s$ close to zero, asymptotically in $s$, and in the medium range for $s$. We start with the estimate close to zero.

\begin{proposition}\label{prop3}
Let $5 \le p < \infty$, $c_p := \frac 1 6 \frac{\Gamma \left(1+ \frac 3 p \right)}{\Gamma \left(1+ \frac 1 p \right)}$. Then for all $0 \le s \le 3$
$$\gamma_p(s) \le \exp(-c_p s^2) \ , \ \gamma_p(0) = 1 \ . $$
\end{proposition}

\begin{proof}
i) Using the series expansion of $\cos(s r)$, we find
\begin{align*}
\gamma_p(s) & = \frac 1 {\Gamma \left(1+ \frac 1 p \right)} \int_0^\infty \cos(s r) \ \exp(-r^p) \ dr \\
& = \frac 1 {\Gamma \left(1+ \frac 1 p \right)} \sum_{n=0}^\infty (-1)^n \frac{s^{2n}}{(2n)!} \int_0^\infty r^{2n} \ \exp(-r^p) \ dr \\
& = \sum_{n=0}^\infty (-1)^n \frac 1 {(2n+1)!} \frac{\Gamma \left(1+ \frac{2n+1} p \right)}{\Gamma \left(1+ \frac 1 p \right)} s^{2n} =: \sum_{n=0}^\infty (-1)^n a_p(n) s^{2n} \ .
\end{align*}
The series expansion of $\exp(-c_p s^2)$ is
$$\exp(-c_p s^2) = \sum_{n=0}^\infty (-1)^n \left(\frac 1 6 \frac{\Gamma \left(1+ \frac 3 p \right)}{\Gamma \left(1+ \frac 1 p \right)} \right)^n \frac{s^{2n}}{n!} =: \sum_{n=0}^\infty (-1)^n b_p(n) s^{2n} \ . $$
The first two terms in both series coincide. Therefore $\gamma_p(s) \le \exp(-c_p s^2)$ will be satisfied for $0< s \le 3$, if for all even integers $n \ge 2$
\begin{equation}\label{eq3.1}
b_p(n) - b_p(n+1) s^2 \ge a_p(n) - a_p(n+1) s^2
\end{equation}
holds. We have for $0 < s \le 3$
$$b_p(n)-b_p(n+1) s^2 = b_p(n) \left(1 - \frac{s^2}{6(n+1)} \frac{\Gamma \left(1+ \frac 3 p \right)}{\Gamma \left(1+ \frac 1 p \right)} \right) > 0 \ ,  $$
$$a_p(n)-a_p(n+1) s^2 = a_p(n) \left(1 - \frac{s^2}{2(n+1)(2n+3)} \frac{\Gamma \left(1+\frac{2n+3} p \right)}{\Gamma \left(1+ \frac{2n+1} p \right)} \right) > 0 \ . $$
Hence \eqref{eq3.1} is equivalent to
\begin{equation}\label{eq3.2}
\frac{b_p(n)}{a_p(n)} :=C_p(n) \ge D_p(n) := \frac{1 - \frac{s^2}{2(n+1)(2n+3)} \frac{\Gamma \left(1+\frac{2n+3} p \right)}{\Gamma \left(1+ \frac{2n+1} p \right)}}{1 - \frac{s^2}{6(n+1)} \frac{\Gamma \left(1+ \frac 3 p \right)}{\Gamma \left(1+ \frac 1 p \right)}} \ .
\end{equation}
For $n = 0$ both values are $1$. \\

ii) We claim that $C_p(n)$ is increasing in $n \ge 2$ for all $p$. Differentiation shows
$$\frac d {dn} (\ln (C_p(n))) = 2 \Psi(2n+2) - \Psi(n+1) + \ln \left(\frac 1 6 \frac{\Gamma \left(1+ \frac 3 p \right)}{\Gamma \left(1+ \frac 1 p \right)} \right) - \frac 2 p \Psi \left(1+ \frac{2n+1} p \right) \ . $$
Let $E(n,p) := \ln (\frac 1 6 \frac{\Gamma(1+ \frac 3 p)}{\Gamma(1+ \frac 1 p)}) - \frac 2 p \Psi(1+ \frac{2n+1} p)$. We will show that $E(n,p)$ is increasing in $p$ for all $n$ (and negative), and then $E(n,p) \ge E(n,2)$ will imply
$$\frac d {dn} (\ln (C_p(n))) \ge 2 \Psi(2n+2) - \Psi(n+1) - 2 \ln 2 - \Psi \left(n+ \frac 3 2 \right) = 0 \ , $$
where the last equality is a consequence of the duplication formula for the Digamma function $\Psi$,
$\Psi(2x) = \frac 1 2 \Psi(x) + \frac 1 2 \Psi(x+\frac 1 2) + \ln 2$, see Abramowitz, Stegum \cite{AS}, 6.3.8. Thus $C_p(n)$ will be increasing in $n$. As for $E(n,p)$, we have
\begin{align*}
p^2 \ & \frac d {dp} E(n,p) \\
& =  2 \Psi \left(1+\frac{2n+1} p \right) + \Psi \left(1+ \frac 1 p \right) - 3 \Psi \left(1+ \frac 3 p \right)  + 2 \frac{2n+1} p \Psi' \left(1+\frac{2n+1} p \right) \\
& \ge   2 \Psi \left(1+ \frac 5 p \right) + \Psi \left(1+ \frac 1 p \right) - 3 \Psi \left(1+ \frac 3 p \right)  \ ,
\end{align*}
using that $\Psi$ is increasing, $\Psi' > 0$, since $\Gamma$ is log-convex. Calculation using \eqref{eq2.4} shows that
$$2 \Psi \left(1+ \frac 5 p \right) + \Psi \left(1+ \frac 1 p \right) - 3 \Psi \left(1+ \frac 3 p \right) = \frac 2 p \sum_{k=1}^\infty \frac{k - \frac 3 p}{(k+\frac 5 p)(k+\frac 1 p)(k+\frac 3 p)} > 0 $$
for $p \ge 3$. In fact, the sum of the first four terms (and thus the full sum) is positive for all $p \ge 2$. Thus $E(n,p)$ is increasing in $p$ for all $n \ge 2$. \\

We next show that $D_p(n,s)$ is decreasing in $n \ge 2$ for all $p\ge 5$ and $s \ge 0$,
$$D_p(n) := \frac{1 - \frac{s^2}{2(n+1)(2n+3)} \frac{\Gamma \left(1+\frac{2n+3} p \right)}{\Gamma \left(1+ \frac{2n+1} p \right)}}{1 - \frac{s^2}{6(n+1)} \frac{\Gamma \left(1+ \frac 3 p \right)}{\Gamma \left(1+ \frac 1 p \right)}} =: \frac{1-B_p(n) s^2}{1-A_p(n) s^2} \ . $$
This is equivalent to
\begin{equation}\label{eq3.3}
(A_p(n)-A_p(n+1))-(B_p(n)-B_p(n+1)) + s^2 (A_p(n+1)B_p(n)-A_p(n)B_p(n+1)) \ge 0 \ .
\end{equation}
The coefficient of $s^2$ is non-negative if and only if
\begin{equation}\label{eq3.4}
(2n+5) \Gamma \left(1+\frac{2n+3} p \right)^2 \ge (2n+3) \Gamma \left(1+\frac{2n+5} p \right) \Gamma \left(1+\frac{2n+1} p \right) \ .
\end{equation}
We have equality for $p=2$ and claim that
$l(p):= 2 \ln (\Gamma(1+\frac{2n+3} p)) - \ln (\Gamma(1+\frac{2n+5} p)) - \ln (\Gamma(1+\frac{2n+1} p))$ is increasing in $p \ge 2$, so that \eqref{eq3.4} will hold. Differentiation shows, using again \eqref{eq2.4},
\begin{align*}
p^2 & \ l'(p) \\
& =   (2n+5) \Psi \left(1+\frac{2n+5} p \right) + (2n+1) \Psi \left(1+\frac{2n+1} p \right) - 2 (2n+3) \Psi \left(1+\frac{2n+3} p \right)  \\
& = \frac 8 p \sum_{k=1}^\infty \frac k {(k+\frac{2n+5} p)(k+\frac{2n+5} p)(k+\frac{2n+5} p)} > 0 \ .
\end{align*}
Therefore the coefficient of $s^2$ in \eqref{eq3.3} is positive. \\
By the log-convexity of $\Gamma$, $\Gamma(1+\frac{2n+3} p)^2 \le \Gamma(1+\frac{2n+1} p) \Gamma(1+\frac{2n+5} p)$. This yields for the first two terms in \eqref{eq3.3}
$$A_p(n)-A_p(n+1) = \frac 1 {6(n+1)(n+2)} \frac{\Gamma \left(1+\frac 3 p \right)}{\Gamma \left(1+\frac 1 p \right)} \ , $$
$$B_p(n)-B_p(n+1) \le \left(\frac 1 {2(n+1)(2n+3)} - \frac 1 {2(n+2)(2n+5)} \right) \frac{\Gamma \left(1+\frac{2n+3} p \right)}{\Gamma \left(1+\frac{2n+1} p \right)} \ . $$
Again by the log-convexity of $\Gamma$, for $x = 1 + \frac{2n+1} p$ and $0 < \theta = \frac 2 p \le 1$
$$\Gamma(x+ \theta) \le \Gamma(x)^{1-\theta} \Gamma(x+1)^\theta = x^\theta \Gamma(x) \ , \ \frac{\Gamma \left(1+\frac{2n+3} p \right)}{\Gamma \left(1+\frac{2n+1} p \right)} \le \left(1+ \frac{2n+1} p \right)^{\frac 2 p} \ . $$
Therefore
\begin{align*}
(A_p(n) - & A_p(n+1)) - (B_p(n) - B_p(n+1)) \\
& \ge \frac 1 {6(n+1)(n+2)} \left(\frac{\Gamma \left(1+\frac 3 p \right)}{\Gamma \left(1+\frac 1 p \right)} - \frac{12 n + 21}{(2n+3)(2n+5)} \left(1+\frac{2n+1} p \right)^{\frac 2 p} \right) \ .
\end{align*}
By Lemma \ref{lem2} (b), $\frac{\Gamma(1+\frac 3 p)}{\Gamma(1+\frac 1 p)} \ge 0.9429$ for all $p \ge 2$. The negative term in the last inequality is maximal for $n=2$, and it is $< 0.9426$ for all $p \ge 5$. Thus \eqref{eq3.3} is satisfied for all $p \ge 5$ and $s \ge 0$. \\

iii) Since $C_p(n)$ is increasing in $n$ and $D_p(n,s)$ is decreasing in $n$ for $p \ge 5$ and $s \ge 0$, the worst case for \eqref{eq3.2} to hold is $n=2$. We have to prove that $C_p(2) \ge D_p(2,s)$ for all $0 \le s \le 3$. Since
\begin{align*}
p^2 \ \frac d {dp} \ln(C_p(2)) & =  5 \Psi \left(1+\frac 5 p \right) + \Psi \left(1+\frac 1 p \right) - 6 \Psi \left(1+\frac 3 p \right)  \\
& = \frac 8 p \sum_{k=1}^\infty \frac k {(k+\frac 5 p)(k+\frac 1 p)(k+\frac 3p)} > 0 \ ,
\end{align*}
$C_p(2)$ is strictly increasing in $p$ with $\lim_{p \to \infty} C_p(2) = \frac 5 3$. Also
$D_p(2,s) = \frac{1-\frac{s^2}{42} \frac{\Gamma(1+\frac 7 p)}{\Gamma(1+\frac 5 p)}}{1-\frac{s^2}{18} \frac{\Gamma(1+\frac 3 p)}{\Gamma(1+\frac 1 p)}}$
is strictly increasing in $p$ for all $p \ge 5$ and $0 \le s \le 3$: $1 \le \frac{\Gamma(1+\frac 7 p)}{\Gamma(1+\frac 5 p)} \le \frac 7 2$ is strictly decreasing and $\frac{\Gamma(1+\frac 3 p)}{\Gamma(1+\frac 1 p)}$ is strictly increasing for $p \ge p_2 \simeq 9.1147$. For $p \in [5,p_2)$, $D_p(2,s)$ is increasing in $p$, since $\frac{\Gamma(1+\frac 7 p)}{\Gamma(1+\frac 5 p)}$ decreases relatively faster than $\frac{\Gamma(1+\frac 3 p)}{\Gamma(1+\frac 1 p)}$. Namely, we have for $5 \le p \le p_2$
$$\frac{\frac d {dp} \frac{\Gamma \left(1+\frac 3 p \right)}{\Gamma \left(1+\frac 1 p \right)}}{\frac d {dp} \frac{\Gamma \left(1+\frac 7 p \right)}{\Gamma \left(1+\frac 5 p \right)}} < \frac 1 4 < 0.29 < \frac 3 7 \frac{1-\frac{s^2}{18} \frac{\Gamma \left(1+\frac 3 p \right)}{\Gamma \left(1+\frac 1 p \right)}}{1-\frac{s^2}{42} \frac{\Gamma \left(1+\frac 7 p \right)}{\Gamma \left(1+\frac 5 p \right)}} \ , $$
where the maximum on the left side occurs for $p=5$ and the minimum on the right side for $s=3$ and $p=p_2$.
Further, $D_p(2,s)$ is increasing in $s$ for all $p$. \\

Define a decreasing sequence $2 \le p_n \le \infty$ as follows: Let $p_1 := \infty$. Since $\frac{11} 7 = D_{p_1}(2,3) < C_{p_1}(2) = \frac 5 3$ and $C_p(2)$ and $D_p(2,3)$ are strictly increasing in $p$ with $C_2(2) = D_2(2,3) = 1$, there is a unique $p_2$, $2 < p_2 < p_1 = \infty$ such that $D_{p_1}(2,3) = C_{p_2}(2)$. Then for all $p_2 \le p < p_1$,
$C_p(2) \ge C_{p_2}(2) = D_{p_1}(2,3) \ge D_p(2,3)$. If $(p_{n-1},p_n)_{n \ge 2}$ have been chosen with $2< p_n < p_{n-1}$ and $C_{p_n}(2) = D_{p_{n-1}}(2,3)$ such that $C_p(2) \ge D_p(2,3)$ for all $p_n \le p < p_{n-1}$, choose $2 \le p_{n+1} < p_n$ with $C_{p_{n+1}}(2) = D_{p_n}(2,3)$. Note that $p_{n+1} > 2$ since $C_2(2) = D_2(2,s)=1$. Then for all $p_{n+1} \le p < p_n$, $C_p(2) \ge C_{p_{n+1}}(2) = D_{p_n}(2,3) \ge D_p(2,3)$. Numerical evaluation yields $p_2 \simeq 8.6$, $p_3 \simeq 5.31$ and $p_4 \simeq 4.76$. Hence for all $p \ge 5$ and $0 \le s \le 3$, $C_p(2) \ge D_p(2,3) \ge D_p(2,s)$, i.e. \eqref{eq3.2} holds. This proves Proposition \ref{prop3}.
\end{proof}

{\bf Remarks.} (1) The estimate is also true for $2 < p \le 5$, though more tedious to prove by the method used. \\
(2) In K\"onig, Koldobsky \cite{KK}, Lemma 4.1, a similar inequality was proven in a different context, where information on the {\it real} zeros of the Fourier transform of a specific function was available. Proposition \ref{prop3} might be shown (easier than here) by a modification of that technique, provided one would know that the (real and complex) zeros $(z_n)_{n \in \N}$ of $\gamma_p$ satisfy $\Re(z_n) \ge 3 + \Im(z_n)$ for all $n$. We have not been able to prove this. \\

\begin{lemma}\label{lem3}
Let $2 \le p < \infty$. Then for any $s > 0$
\begin{equation}\label{eq3.5}
\left| \frac{\sin(s)} s - \Gamma \left(1+ \frac 1 p \right) \gamma_p(s) \right| \le \frac {1.016} p \ ,
\end{equation}
hence
$$\left| \gamma_p(s) \right| \ge \frac 1 {\Gamma \left(1+ \frac 1 p \right)} \left( \left| \frac{\sin(s)} s \right| - \frac{1.016} p \right) \ . $$
\end{lemma}

\begin{proof}
We have
\begin{align*}
\Big| \frac{\sin(s)} s & - \Gamma \left(1+\frac 1 p \right) \gamma_p(s) \Big| \\
& = \Big| \int_0^1 \cos(s r) \ (1- \exp(-r^p)) \ dr - \int_1^\infty \cos(s r) \ \exp(-r^p) \ dr \Big| \\
& \le \int_0^1 (1-\exp(-r^p)) \ dr + \int_1^\infty \exp(-r^p) \ dr \\
& = 1- \int_0^\infty \exp(-r^p) \ dr + 2 \int_1^\infty \exp(-r^p) \ dr \\
& = 1 - \Gamma \left(1+\frac 1 p \right) + \frac 2 p \int_1^\infty t^{\frac 1 p -1} \ \exp(-t) \ dt =: \phi(p) \ .
\end{align*}
We claim that $\psi(p) := p \phi(p)$ is increasing in $p$. We find
$$\psi'(p) = 1 - \Gamma \left(1+\frac 1 p \right) + \Gamma \left(1+\frac 1 p \right) \frac 1 p \Psi \left(1+ \frac 1 p \right) - \frac 2 {p^2} \int_1^\infty t^{\frac 1 p -1} \ \ln(t) \ \exp(-t) \ dt \ . $$
It follows from \eqref{eq2.4} that $\Psi(1+\frac 1 p) \ge - \gamma + 4(1- \ln 2) \frac 1 p > - \gamma + \frac 6 5 \frac 1 p$ for all $p \ge 2$. Further
$2 \int_1^\infty t^{-\frac 1 2} \ \ln(t) \ \exp(-t) \ dt < \frac 3 {10}$. This implies
\begin{equation}\label{eq3.6}
\psi'(p) \ge 1 - \Gamma \left(1+\frac 1 p \right) \left(1+ \frac{\gamma} p - \frac 6 5 \frac 1 {p^2} \right) - \frac 3 {10} \frac 1 {p^2} \ .
\end{equation}
For all $p \ge 2$, $\Gamma(1+\frac 1 p) \le 1 - \frac {\gamma} p + \frac 1 {p^2}$, whereas for $2 \le p \le 4$ the better estimate
$\Gamma(1+\frac 1 p) \le 1 - \frac {\gamma} p + \frac 5 6 \frac 1 {p^2}$ holds. Using these estimates in \eqref{eq3.6}, calculation shows that $\psi'(p) > 0$ for all $p \ge 2$. Therefore $\psi(p) = p \phi(p) \le \lim_{p \to \infty} \psi(p) = \gamma + 2 \int_1^\infty t^{-1} \ \exp(-t) \ dt = \gamma + 2 Ei(1,1) < 1.016$ and \eqref{eq3.5} follows. Here $Ei$ denotes the exponential integral function.
\end{proof}

For $15 \le p < \infty$, we need a better, slightly more complicated estimate for $\gamma_p(s)$ in the medium range for $s>0$.

\begin{lemma}\label{lem4}
For $15 \le p < \infty$ and $s>0$ let
$$\Phi_p(s) := \left(\frac {14}{19} \frac{\sin(s)} s + \frac 5 {19} \frac{\sin((1-\frac 2 p)s)} s \right) \frac{\sin(\frac s p)}{\frac s p} \ . $$
Then for all $s > 0$
$$\left| \Phi_p(s) - \Gamma \left(1+\frac 1 p \right) \gamma_p(s) \right| \le \frac 1 {N p} \ , $$
where $N = 7.857$ for $175 < p < \infty$, $N = 8.003$ for $26 \le p \le 175$ and $N = 8.62$ for $15 \le p \le 26$. Hence
$$\left| \gamma_p(s) \right| \ge \frac 1 {\Gamma \left(1+\frac 1 p \right)} \Big( \left| \Phi_p(s) \right| - \frac 1 {N p} \Big) \ . $$
The function $\Phi_p$ may be also written as
$$\Phi_p(s) = \frac 1 {19} \sqrt{221+140 \ \cos(2 \frac s p)} \ \frac{\sin(s - \alpha_p(s))} s \ \frac{\sin(\frac s p)}{\frac s p} \ , $$
where $\alpha_p(s) = \arctan \left( \frac{5 \sin(2 \frac s p)}{14 + 5 \cos(2 \frac s p)} \right)$.
\end{lemma}

\begin{proof}
In Lemma \ref{lem3} we approximated $e_p(r):=\exp(-r^p)$ just by $1_{[0,1]}(r)$. We now use a better piecewise linear approximation $k_p \in C[0,\infty)$. Define
$$k_p(r) := \left \{ \begin{array} {c@{\;}}
1 \quad \quad \quad \quad \quad \ r \in I_{1,p} := [0, 1-\frac 3 p] \\
\frac 5 {38} p  (1-r) + \frac {23}{38} \quad   r \in I_{2,p} := (1-\frac 3 p,1-\frac 1 p) \\
\frac 7 {19} p (1-r) + \frac 7 {19} \quad   r \in I_{3,p} := [1-\frac 1 p,1+\frac 1 p] \\
\ 0 \quad \quad \quad \quad \quad \; \ r  \in I_{4,p} := (1+\frac 1 p,\infty) \end{array} \right \} \ . $$
We claim that $\int_0^\infty |e_p(r) - k_p(r)| \ dr < \frac 1 {Np}$, with the value $N$ given in the statement of Lemma \ref{lem4}. \\
Let $\phi_1(p):= \int_{I_{1,p} \cup I_{4,p}} |e_p(r) -k_p(r)| \ dr = \int_{I_{1,p}} (1- \exp(-r^p)) \ dr + \int_{I_{4,p}} \exp(-r^p) \ dr$.
Then $\psi_1(p) := p \phi_1(p)$ is increasing in $p$, similarly as in the proof of Lemma \ref{lem3}, with
\begin{align*}
\psi_1(p) & \le \lim_{p \to \infty} \psi_1(p) = \gamma- 3 + \int_{\exp(-3)}^\infty \frac 1 r \ \exp(-r) \ dr  \\
& + \int_{\exp(1)}^\infty \frac 1 r \ \exp(-r) \ dr =: \psi_1(\infty) \le 0.06791 \ .
\end{align*}
Define $\phi_2(p) := \int_{I_{2,p} \cup I_{3,p}} |e_p(r) - k_p(r)| \ dr$ and $\psi_2(p) := p \phi_2(p)$. The substitution $x = 1 - \frac p {10} (1-r)$ transforms the intervals $I_{2,p}$ and $I_{3,p}$ onto $J_1:= [\frac 7 {10} , \frac 9 {10}]$ and $J_2 :=[\frac 9 {10},\frac {11}{10}]$. We get
\begin{align*}
\psi_2(p) & = 10 \int_{J_1} \left| \exp(-[1-\frac {10} p (1-x)]^p) - (\frac{73}{38} - \frac{25}{19} x) \right| \ dx \\
& + 10 \int_{J_2} \left| \exp(-[1-\frac {10} p (1-x)]^p) - (\frac{77}{19} - \frac{70}{19} x) \right| \ dx \ .
\end{align*}
Denote $l_p(x) := \exp(-[1-\frac {10} p (1-x)]^p))$, which is decreasing in $p$ for all $x$ in $J := [\frac 7 {10}, \frac{11}{10}]$, since $(1+\frac y p)^p$ is increasing in $p$ for $|y| < p$. Note that $l_p(x)$ converges uniformly on $J$ to $l_\infty(x) = \exp(-\exp(-10 \ (1-x)))$. For $x=1$ we have equality $l_p(1) = l_\infty(1) = \frac 1 e$.
With $\lambda_1(x) = \frac{73}{38} - \frac{25}{19} x$ and $\lambda_2(x) = \frac{77}{19} - \frac{70}{19} x$ we have
$$\psi_2(p) = 10 \int_{J_1} \left| l_p(x) - \lambda_1(x) \right| \ dx + 10 \int_{J_2} \left| l_p(x) - \lambda_2(x) \right| \ dx \ . $$
Let $q < p \le \infty$. Then $l_q(x) \ge l_p(x)$ and the triangle inequality yields
$$|\psi_2(p) - \psi_2(q)| \le 10 \int_J (l_q(x) - l_p(x)) \ dx =: \psi_3(p,q) \ . $$
The claim $\int_0^\infty |e_p(r) - k_p(r)| \ dr < \frac 1 {Np}$ means that
$\psi(p) := \psi_1(p) + \psi_2(p) < \frac 1 N$. Since $\psi_1$ is increasing in $p$ and $\psi_3(p,q)$ is increasing in $p$ for any fixed $q$, $q < p$, we find for all $26 \le p < \infty$ that
$$\psi(p) \le \psi_1(\infty) + \psi_2(26) + \psi_3(p,26) \le \psi_1(\infty) + \psi_2(26) + \psi_3(\infty,26) \ . $$
We know already $\psi_1(\infty) < 0.06791$. Numerical integration yields $\psi_2(26) < 0.05675$ and $\psi_3(\infty,26) < 0.01993$, so that $\psi(p) < 0.1446 < \frac 1 N$ with $N = 6.9$. To find the better values of $N$ stated in Lemma \ref{lem4}, one has to apply the previous procedure to a sequence of points, namely
$p_1 = \infty$, $p_2 = 600$, $p_3 = 175$, $p_4 = 120$, $p_5 = 73$, $p_6 = 43$ and $p_7 = 26$ and use for $p_{j+1} \le p \le p_j$ for $j = 1 \etc 6$ that
$$\psi(p) \le \psi_1(p_j) + \psi_2(p_{j+1}) + \psi_3(p,p_{j+1}) \le \psi_1(p_j) + \psi_2(p_{j+1}) + \psi_3(p_j,p_{j+1}) =: C_j \ . $$
Numerical integration of $\psi_1(p_j)$, $\psi_2(p_{j+1})$ and $\psi_3(p_j,p_{j+1})$ for $j=1 \etc 6$ yields that $C_1, C_2 < \frac 1{7.857}$ and $C_j < \frac 1 {8.003}$ for $j= 3 \etc 6$. Hence $\psi(p) < \frac 1 N$ for all $p$ with $26 \le p < \infty$. The scheme with several points $p_j$ yields a better result, since $\psi_1(p)$ is strictly increasing in $p$. We have e.g. $\psi_1(43) < 0.0602 < 0.0679 < \psi_1(\infty)$.
For $15 \le p \le 26$, let $p_7=26$, $p_8=22.3$, $p_9=18.5$ and $p_{10} = 15$. Use $\psi_2(p_j)$ instead of $\psi_2(p_{j+1})$ to find for all $p$ with $p_{j+1} \le p \le p_j$ for $j=7, 8, 9$ that
$$\psi(p) \le \psi_1(p_j) + \psi_2(p_j) + \psi_3(p,p_{j+1}) \le \psi_1(p_j) + \psi_2(p_j) + \psi_3(p_j,p_{j+1}) =: D_j \ . $$
Numerical evaluation gives $D_j < \frac 1 {8.62}$ for $j=7, 8, 9$. \\

Thus $\int_0^\infty |e_p(r) - k_p(r)| \ dr < \frac 1 {Np}$ for all $p \ge 15$. This implies
$$\left| \int_0^\infty \cos(s r) \ \exp(-r^p) \ dr - \int_0^\infty \cos(s r) \ k_p(r) \ dr \right| \le \int_0^\infty |e_p(r) -k_p(r)| \ dr < \frac 1 {Np} \ . $$
Explicit integration yields
$$\int_0^\infty \cos(s r) \ k_p(r) \ dr = \left(\frac {14}{19} \frac{\sin(s)} s + \frac 5 {19} \frac{\sin((1-\frac 2 p)s)} s \right) \frac{\sin(\frac s p)}{\frac s p} = \Phi_p(s) \ . $$
To verify the alternative form of $\Phi_p(s)$ with only one $\frac{\sin(s-\alpha)} s$ term, let $y:=2 \frac s p$ and define $\alpha = \arctan \left( \frac{5 \sin(y)}{14 + 5 \cos(y)} \right)$. Then $\tan(\alpha) =  \frac{5 \sin(y)}{14 + 5 \cos(y)}$, which implies $\cos(\alpha) = \frac{14 + 5 \cos(y)}{\sqrt{221 + 140 \ \cos(y)}}$, $\sin(\alpha) = \frac{5 \sin(y)}{\sqrt{221 + 140 \ \cos(y)}}$. Therefore with $(1-\frac 2 p) s = s - y$
\begin{align*}
\frac{14}{19} \sin(s) & + \frac 5 {19} \sin(s - y) = \frac 1 {19} \left( \sin(s) (14 + 5 \cos(y)) - 5 \cos(s) \sin(y) \right) \\
& = \frac 1 {19} \sqrt{221+140 \ \cos(y)} \left( \sin(s) \cos(\alpha) - \cos(s) \sin(\alpha) \right) \\
& = \frac 1 {19} \sqrt{221+140 \ \cos(y)} \sin(s - \alpha) \ .
\end{align*}
\end{proof}

{\bf Remarks.} (1) Let $\beta(y) := \arctan \left( \frac{5 \sin(y)}{14 + 5 \cos(y)} \right)$. Then $\beta'(y) = 5 \frac{14 \ \cos(y) + 5}{221 +140 \ \cos(y)}$. This is zero for $y_{\pm} = \pi \pm \arccos(\frac 5 {14})$ and $\beta$ has a maximum in $y_- \simeq 1.936$, $\beta(y_-) = \arctan(\frac 5 3 \frac 1 {\sqrt{19}}) \simeq 0.3653$ and a minimum in $y_+ \simeq 4.347$, $\beta(y_+) = - \arctan(\frac 5 3 \frac 1 {\sqrt{19}})$. Hence $|\beta(y)| \le 0.3653$ for all $y$. The function $\beta$ is increasing in $[0,y_-]$ and decreasing in $[y_-,y_+]$. It is positive in $(0,\pi)$ and negative in $(\pi,2 \pi)$. For $\alpha_p(s) = \beta(2 \frac s p)$ we have that $\alpha_p$ is increasing in $s$ and decreasing in $p$, if $0 \le 2 \frac s p \le y_-$. The converse holds for $y_- \le 2 \frac s p \le y_+$. We have $-\frac 5 9 \le \beta'(y) \le \frac 5 {19}$, with equality for the upper estimate at $y=0$ and for the lower estimate at $y = \pi$. \\

(2) Lemma \ref{lem4} holds for $6 \le p \le 15$, too, with $N = 8$. This implies $\gamma_p(s) \ge 0$ for all $0 \le s \le 3$ and $p \ge 6$. \\

\begin{corollary}\label{cor2}
Assume $p \ge 15$, $0 \le s \le \frac {y_-} 2 p$ and $s - \alpha_p(s) \ge [\frac s \pi] \pi$. The assumption is satisfied e.g. for $s \in [0,\pi]$, $s \in [3.255,2 \pi]$, for $s \in[6.501,3 \pi]$ and for $s \in [9.73,4 \pi]$. Put
$$\Psi_p(s) := \frac 1 {19} \sqrt{221+140 \ \cos(2 \frac s p)} \ \frac{\min( |\sin(s-\alpha_p(s))|, |\sin(s)| )} s \ \frac{\sin(\frac s p)}{\frac s p} - \frac 1 {N p} \ ,  $$
with $N$ as in Lemma \ref{lem4}. Then $\Gamma(1 + \frac 1 p) |\gamma_p(s)| \ge \Psi_p(s)$, and for each $s$, $\Psi_p(s)$ is increasing in $p$.
\end{corollary}

\begin{proof}
The factors $\frac 1 {19} \sqrt{221+140 \ \cos(2 \frac s p)}$ and $\frac{\sin(\frac s p)}{\frac s p}$ are strictly increasing in $p$ since $\frac s p$ is decreasing. Also $- \frac 1 p$ is increasing in $p$. If $s < \frac {y_-} 2 p$, $\alpha_p(s)$ is positive and decreasing in $p$, $s - \alpha_p(s) < s$ is increasing in $p$. In intervals where $|\sin(x)|$ is increasing, $|\sin(s-\alpha_p(s))| \le |\sin(s)|$ and $\min( |\sin(s-\alpha_p(s))|, |\sin(s)| ) = |\sin(s-\alpha_p(s))|$. This is increasing in $p$ under our assumption that $s - \alpha_p(s) \ge [\frac s \pi] \pi$, and hence $\Psi_p(s)$ is increasing in $p$, too. In intervals where $|\sin(x)|$ is decreasing, $|\sin(s-\alpha_p(s))| \ge |\sin(s)|$ and
$\min( |\sin(s-\alpha_p(s))|, |\sin(s)| ) = |\sin(s)|$ is constant in $p$. \\
We have $s - \alpha_p(s) \ge 0$, since $\alpha_p(s) \le \frac 5 {19} \frac 2 p s \le s$. Note that for $p \ge 15$ and $0 \le s < 14$, $s < \frac{y_-}2 p$ is satisfied. For $0 \le s \le \frac {y_-} 2 p$, $s-\alpha_p(s)$ is increasing in $s$ and in $p$, since $\frac{\partial}{\partial p} (s-\alpha_p(s)) = \frac{10s}{p^2} \frac{14 \ \cos(2 \frac s p) + 5}{221 +140 \ \cos(2 \frac s p)} > 0$ and $\frac{\partial}{\partial s} (s-\alpha_p(s)) = 1 - \frac 2 p \beta'(2 \frac s p) \ge 1 - \frac{10}{19 p} >0$. This means that the valid inequality $3.255 - \alpha_{15}(3.255) > \pi$ implies $s - \alpha_p(s) > \pi$ for all $p \ge 15$ and $s \in [3.255,2 \pi]$. Similarly $6.501 - \alpha_p(6.501) > 2 \pi$ and $9.73 - \alpha_p(9.73) > 3 \pi$ imply the other two claims when the assumption is satisfied for $s$.
\end{proof}

\begin{corollary}\label{cor3}
Let $p \ge 15$ and $0 \le s \le p$. For $n \in \N$ and $I_n := [n \pi, (n+1) \pi]$ put $y_n := \max_{s \in I_n} \left| \frac{\sin(s)} s \right|$. Then
$$|\gamma_p(s)| \le \frac{y_n + \frac 1 {120}}{\Gamma(1+ \frac 1 {15})} \ , \ s \in I_n \ . $$
We have $y_1 \simeq 0.21723$ with maximum at $s_1 \simeq 4.493$, $y_2 \simeq 0.12827$ with maximum at $s_2 \simeq 7.725$. Hence $|\gamma_p(s)| \le 0.2336$ for $s \in I_1$ and $|\gamma_p(s)| \le 0.1416$ for $s \in I_2$ for $p \ge 15$. Replacing $15$ here by $p_0 \simeq 26.265$ yields $|\gamma_p(s)| \le 0.2267$ for $s \in I_1$ and $|\gamma_p(s)| \le 0.1360$ for $s \in I_2$ for any $p \ge p_0$.
\end{corollary}

\begin{proof}
For $s \le p$, $2 \frac s p < \pi$ and $\alpha_p(s) \ge 0$. Hence for $s - \alpha_p(s) \in I_n $
$$\frac{|\sin(s - \alpha_p(s))|} s = \frac{s - \alpha_p(s)} s \frac{|\sin(s - \alpha_p(s))|}{s - \alpha_p(s)} \le y_n \ . $$
This is also true, if $s \in I_n$, but $s-\alpha_p(s) < n \pi$: We have $\alpha_p(s) < 0.37 < \frac 2 5$, $s- \alpha_p(s) > n \pi - \frac 2 5$, with
$$\frac{|\sin(n \pi - \frac 2 5)|}{n \pi - \frac 2 5} = \frac{\sin(\frac 2 5)}{n \pi - \frac 2 5} < \frac 1 {(n+\frac 1 2) \pi} \le y_n \ . $$
Therefore for $s \in I_n$, $\Gamma(1 + \frac 1 p) |\gamma_p(s)| \le \Phi_p(s) + \frac 1 {N p} \le y_n + \frac 1 {N p}$, $|\gamma_p(s)| \le \frac{y_n + \frac 1 {N p}}{\Gamma(1+\frac 1 p)}$, which is maximal for $p=15$. The upper estimate for $I_1$ easily extends to $[3,2 \pi]$.
\end{proof}

\begin{lemma}\label{lem5}
(a) Let $2 \le p < \infty$. The for any $s >0$
$$\left| \Gamma(1+\frac 1 p) \gamma_p'(s) - \frac{s \cos(s) - \sin(s)}{s^2} \right| \le \frac {1.016} p $$
and
$$\left|\gamma_p'(s) \right| \le \frac{1.064}{\Gamma \left(1+\frac 1 p \right)} \left( \frac 1 s + \frac 1 p \right) \ . $$
For $s > \pi$ the constant $1.064$ in the second inequality may be replaced by $1.016$. \\

(b) Let $15 \le p < \infty$. Then $\gamma_p$ is strictly convex in $[2.5,5.5]$ with $\gamma(2.5) > 0$ and $\gamma_p(5.5) < 0$. In $[5.5,7]$ $\gamma_p$ is strictly increasing with $\gamma_p(7)>0$. Further, $\gamma_p$ is strictly concave in $[7,8.5]$ and strictly decreasing in $[8.5,10]$ with $\gamma_p(10)< 0$.
\end{lemma}

\begin{proof}
(a) We have $ \gamma_p'(s) = \frac{-1}{\Gamma(1+\frac 1 p)} \int_0^\infty \sin(s r) \ r \ \exp(-r^p) \ dr$ and
\begin{align*}
\Big| \int_0^\infty \sin(s r) & \ r \ (1_{[0,1]}(r) - \exp(-r^p)) \ dr \Big| \\
& \le \int_0^1 r \ (1- \exp(-r^p)) \ dr + \int_1^\infty r \ \exp(-r^p) \ dr \\
& = \frac 1 2 -\int_0^\infty r \ \exp(-r^p) \ dr + 2 \int_1^\infty r \ \exp(-r^p) \ dr \\
& = \frac 1 2 \left(1 - \Gamma \left(1 + \frac 2 p \right) \right) + \frac 2 p \int_1^\infty t^{\frac 2 p -1} \ \exp(-t) \ dt =: \phi(p) \ .
\end{align*}
As in the proof of Lemma \ref{lem3}, $\psi(p) := p \phi(p)$ is increasing in $p$, with $\psi(p) \le \gamma + 2 Ei(1,1) < 1.016$. Hence, using
$\int_0^1 \sin(s r) \ r \ dr = \frac{\sin(s) - s \cos(s)}{s^2}$,
$$\left| \Gamma \left(1 + \frac 1 p \right) \gamma_p'(s) - \frac{s \cos(s) - \sin(s)}{s^2} \right| \le \frac{1.016} p \ . $$
For $0 \le s \le \pi$, $h(s) := |\frac{s \cos(s) - \sin(s)}{s^2}|$ is bounded by $\frac{1.064} s$, with maximum of $s h(s)$ at $s_0 \simeq 2.744$. For $s > \pi$, $h(s)$ is bounded by $\frac{1.014} s$, the extrema occurring for $s$ satisfying $2 s \cos(s) + (s^2-2) \sin(s) = 0$, $s \simeq n \pi - \frac 2 {n \pi}$. Therefore $\Gamma(1+ \frac 1 p) |\gamma_p'(s)| \le 1.064 (\frac 1 s + \frac 1 p)$. \\

(b) We have $\gamma_p''(s) = \frac{-1}{\Gamma(1 + \frac 1 p)} \int_0^\infty \cos(s r) \ r^2 \ \exp(-r^p) \ dr$. Similarly as above we find
\begin{align*}
\Big| \int_0^\infty & \cos(s r) \ r^2 \ (1_{[0,1]}(r) - \exp(-r^p)) \ dr \Big| \\
& \le \int_0^1 r^2 \ (1_{[0,1]}(r) - \exp(-r^p)) \ dr + \int_1^\infty \exp(-r^p) \ dr \\
& = \frac 1 3 (1 - \Gamma(1 + \frac 3 p)) + \frac 2 p \int_1^\infty t^{\frac 3 p -1} \ \exp(-t) \ dt =: \tilde{\phi}(p) \ .
\end{align*}
Again $\tilde{\psi}(p) := p \tilde{\phi}(p)$ is increasing in $p$, with $\tilde{\psi}(p) \le \gamma + 2 Ei(1,1) < 1.016$. Hence
$\left| \gamma_p''(p) + \int_0^1 \cos(s r) \ r^2 \ dr \right| < \frac{1.016} p$. Since $\int_0^1 \cos(s r) \ r^2 \ dr = \frac{(s^2-2) \sin(s) + 2 s \cos(s)}{s^3} =: k(s)$ with $k(s) \le \frac{-1}{14}$ for $s \in [2.5,5.5]$, as a standard investigation shows, $\gamma_p''(s) \ge \frac 1 {14} - \frac{1.016} p > 0$ for $p \ge 15$ and $\gamma_p$ is strictly convex in $[2.5,5.5]$. In $[5.5,7]$, the function $l(s) = \frac{s \cos(s) - \sin(s)}{s^2}$ is positive with $l(s) > \frac 1 {11}$. Hence by part (a), $\Gamma \left(1 + \frac 1 p \right) \gamma_p'(s) > \frac 1 {11} - \frac{1.016} p > 0$, i.e. $\gamma_p$ is strictly increasing there. Similarly we have in $[7,8.5]$ that $k(s) > \frac 1 {14}$, $\gamma_p''(s) < 0$ and in $[8.5,10]$ that $l(s) < \frac{-1}{14}$, $\gamma_p'(s) < 0$. \\
The estimate $\left| \Gamma \left(1+ \frac 1 p \right) \gamma_p(s) - \frac{\sin(s)} s \right| \le \frac {1.016} p$ of Lemma \ref{lem3} yields $\gamma_p(2.5)>0$, $\gamma_p(5.5) < 0$, $\gamma_p(7) > 0$ and $\gamma_p(10) < 0$, since $\frac{\sin(2.5)}{2.5} > \frac 1 5$, $\frac{\sin(5.5)}{5.5} <  \frac {-1} 8$, $\frac{\sin(7)}{7} > \frac 1 {11}$ and $\frac{\sin(10)}{10} <  \frac{-1}{14}$.
\end{proof}

P\'olya \cite{Po} proved an asymptotic estimate for $\gamma_p(s)$, see also Koldobsky \cite{K}, Lemma 2.28, if $p>2$ is not an even integer. We need this with an error estimate.

\begin{proposition}\label{prop4}
Let $5 \le p < \infty$, $p \notin 2 \N$. Then for any $s \ge \frac 2 3 \frac{p^3}{|\sin(\frac{\pi p} 2)|^{\frac 1 p}}$
\begin{equation}\label{eq3.9}
\left| \Gamma \left(1+\frac 1 p \right) s^{p+1} \gamma_p(s) - \Gamma(p+1) \sin (\frac{\pi p} 2) \right| < \frac 1 2 \Gamma(p+1) \left|\sin (\frac{\pi p} 2) \right| \ .
\end{equation}
For $p \ge 10$, $\frac 2 3$ here may be replaced by $\frac 5 8$.
\end{proposition}

\begin{proof}
As in \cite{Po}, we integrate $\gamma_p(s)$ by parts, substitute $z = s^p t^p$ and put $\delta := s^{-p}$ to find
\begin{align*}
\Gamma(1+\frac 1 p) s^{p+1} \gamma_p(s) & = s^{p+1} \int_0^\infty \cos(st) \ \exp(-t^p) \ dt \\
& = p s^p \int_0^\infty \sin(s t) \ t^{p-1} \ \exp(-t^p) \ dt \\
& = p s^p \Im \left( \int_0^\infty \exp(i s t) \ t^{p-1} \ \exp(-t^p) \ dt \right) \\
& = \Im \left( \int_0^\infty \exp(i z^{\frac 1 p}) \ \exp(-\delta z) \ dz \right) \ .
\end{align*}
Fix $0 < \theta \le \frac \pi 2$. By analyticity, we may change the path of integration from $[0,\infty)$ to $\exp(i \theta) [0,\infty)$, since for large $R>0$ and $z = R \exp(i \phi)$
\begin{align*}
\Big| \int_0^\theta & \exp(i R^{\frac 1 p} \exp(i \frac \phi p)) \exp(-\delta R \exp(i \phi)) \ d \phi \Big| \\
& \le \int_0^\theta \exp(- R^{\frac 1 p} \sin(\frac \phi p) - \delta R \cos(\phi)) \ d \phi \ , \ \cos(\phi) \ge 0 \\
& \le \int_0^\theta \exp(-R^{\frac 1 p} \sin(\frac \phi p)) \ d \phi \le \int_0^\theta \exp(- \frac 2 {\pi p} R^{\frac 1 p} \phi) \ d \phi \\
& = \frac \pi 2 \frac p {R^{\frac 1 p}} \left( 1 - \exp(- \frac 2 \pi \frac \theta p R^{\frac 1 p} ) \right) \le \frac \pi 2 \frac p {R^{\frac 1 p}} \ ,
\end{align*}
which tends to zero for $R \to \infty$. Thus for $z = r \exp(i \theta)$
\begin{align}\label{eq3.10}
\Gamma \left(1+\frac 1 p \right) & s^{p+1} \gamma_p(s) \nonumber \\
& = \Im \left(\exp(i \theta) \int_0^\infty \exp \left(i r^{\frac 1 p} \ \exp(i \frac \theta p) \right) \ \exp(-\delta r \cos(\theta)) \ \exp(-i \delta r \sin(\theta)) \ dr \right) \ .
\end{align}
As in P\'olya \cite{Po}, we find for $s \to \infty$, $\delta \to 0$
\begin{align}\label{eq3.11}
\Gamma \left(1+\frac 1 p \right) s^{p+1} \gamma_p(s) & \to \Im \left(\exp(i \theta) \int_0^\infty \exp \left(i r^{\frac 1 p} \exp \left(i \frac \theta p \right) \right) \ dr \right) \nonumber \\
& = \Im \left(\exp(i \frac{\pi p} 2) \int_0^\infty \exp(-r^{\frac 1 p}) \ dr \right) = \Gamma(p+1) \ \sin(\frac{\pi p} 2) \ ,
\end{align}
where we used analyticity again to replace $r \exp(i \theta)$ by $r \exp(i \frac{\pi p} 2)$. \\

To get the error estimate in \eqref{eq3.9}, we have to estimate the difference between the integrals in \eqref{eq3.10} and \eqref{eq3.11}. For
$a= \exp(-\delta r \cos(\theta))$, $\psi = - \delta r \sin(\theta)$, we find using $\cos(\theta) \ge 0$, $\sin(x) \le x$ and $0 \le 1 - \exp(-x) \le x$ for $x \ge 0$
\begin{align*}
| 1 - a \exp(i \psi) |^2 & = (1-a \cos(\psi))^2+ a^2 \sin(\psi)^2 = (1-a)^2 + 4 a \sin \left(\frac \psi 2 \right)^2 \\
& = \left(1 - \exp(-\delta r \cos(\theta)) \right)^2 + 4 \exp \left(-\delta r \cos(\theta) \right) \sin \left(\frac{\delta r \sin(\theta)} 2 \right)^2 \\
& \le (\delta r \cos(\theta))^2 + 4 \left(\frac{\delta r \sin(\theta)} 2 \right)^2 = (\delta r)^2 \ ,
\end{align*}
i.e. $| 1 - a \exp(i \psi)| \le \delta r$. This implies
\begin{align*}
\Big| \Gamma & \left(1+\frac 1 p \right) s^{p+1} \gamma_p(s) - \Gamma(p+1) \sin \left(\frac{\pi p} 2 \right) \Big| \\
& \le \int_0^\infty \delta r \Big| \exp \left(i r^{\frac 1 p} \exp\left(i \frac \theta p \right) \right) \Big| \ dr = \delta \int_0^\infty r \exp \left(-r^{\frac 1 p} \sin \left(\frac \theta p \right) \right) \ dr =: I \ .
\end{align*}
Substituting $u = r^{\frac 1 p} \sin(\frac \theta p)$, i.e. $r = (\frac u {\sin(\frac \theta p)})^p$, we get with $\delta= s^{-p}$
\begin{align*}
I & = \delta \int_0^\infty p \frac{u^{2p-1}}{\sin(\frac \theta p)^{2p}} \ \exp(-u) \ du \\
& = \frac{\delta p}{\sin(\frac \theta p)^{2p}} \Gamma(2p) = \frac 1 2 \frac{\Gamma(2p+1)}{\sin(\frac \theta p)^{2p}} \frac 1 {s^p} \ .
\end{align*}
This holds, in particular, for the choice $\theta=\frac \pi 2$. To prove \eqref{eq3.9}, it suffices that
$$\frac 1 2 \frac{\Gamma(2p+1)}{\sin(\frac \pi {2p})^{2p}} \frac 1 {s^p} < \frac 1 2 \Gamma(p+1) \ |\sin(\frac{\pi p} 2)| \ , $$
which means
$$ s > \left(\frac{\Gamma(2p+1)}{\Gamma(p+1)} \right)^{\frac 1 p} \frac 1 {\sin(\frac \pi {2p})^2} \frac 1 {|\sin(\frac{\pi p} 2)|^{\frac 1 p}} \ . $$
We claim that $h_1(p):=\left(\frac 1 p \frac{\Gamma(2p+1)}{\Gamma(p+1)} \right)^{\frac 1 p}$ is decreasing in $p$. We have
$$(\ln h_1)'(p) = - \frac 1 {p^2} k(p) \ , \ k(p) = \ln \left( \frac{\Gamma(2p+1)}{\Gamma(p+1)} \right) + p \left(1 + \Psi(p+1) - 2 \Psi(2p+1) \right) \ . $$
To show that $k$ is positive, note that $k(1)= \gamma + \ln 2 -1 > \frac 1 4 >0$. Thus it suffices to show that $k'(p) > 0$. We find
$k'(p) = 1 + p (\Psi'(p+1)-4 \Psi'(2p+1))$. By the duplication formula for $\Psi$, see Abramowitz, Stegun \cite{AS}, 6.3.8, $2 \Psi(2 z) = \Psi(z) + \Psi(z + \frac 1 2) + 2 \ln 2$. Hence $4 \Psi'(2 z) - \Psi'(z + \frac 1 2) = \Psi'(z)$. For $z = p + \frac 1 2$ we get $4 \Psi'(2p+1)-\Psi'(p+1) = \Psi'(p+\frac 1 2)$. Thus
$k'(p) = 1 - p \Psi'(p+\frac 1 2)$. By \eqref{eq2.4}, $\Psi'(p+\frac 1 2) = \sum_{k=1}^\infty \frac 1 {(k+p-\frac 1 2)^2}$. Let $y := p - \frac 1 2$. Then for any $k \in \N$ we have $\int_{k-\frac 1 2}^{k+\frac 1 2} \frac{dx}{(x+y)^2} = \frac 1 {(k+y)^2-\frac 1 4} >  \frac 1 {(k+y)^2}$ and
$\sum_{k=1}^\infty \frac 1 {(k+p-\frac 1 2)^2} < \int_{\frac 1 2}^\infty \frac{dx}{(x+y)^2} = \frac 2 {1+2y} = \frac 1 p$.
Therefore $k'(p) = 1 - p \sum_{k=1}^\infty \frac 1 {(k+p-\frac 1 2)^2} > 0$. Hence $h_1$ is decreasing in $p$. It is easily verified that $h_2(p) := p \sin(\frac \pi {2p})$ is increasing in $p$. Let $h(p) := \frac{h_1(p)}{h_2(p)^2}$. Then $h(p)$ is decreasing in $p$ and
$\left(\frac{\Gamma(2p+1)}{\Gamma(p+1)} \right)^{\frac 1 p} \frac 1 {\sin(\frac \pi {2p})^2} \le h(p) p^3$.
This proves Proposition \ref{prop4} for $s > h(p) \frac{p^3}{|\sin(\frac{\pi p}2)|^{\frac 1 p}}$. We have $h(5) \le \frac 2 3$, $h(10) \le \frac 5 8$, $h(60) \le \frac 3 5$ and $\lim_{p\to \infty} h(p) = \frac{16}{\pi^2 e} < 0.597$.
\end{proof}

For even integers $p \in 2 \N$, the decay of $\gamma_p(s)$ is not polynomial but exponential in nature. Boyd \cite{Bo} proved the following asymptotic expansion: Let $p \in 2 \N$, $p \ge 4$. Then for large $s>0$
\begin{align}\label{eq3.12}
\Gamma \left(1+\frac 1 p \right) & \gamma_p(s) \simeq \sqrt{\frac{2 \pi} {p-1}} \frac 1 {p^{\frac 1 {2(p-1)}} s^{\frac{p/2-1}{p-1}} } \ \exp \Big(-(p-1) \ \sin \left(\frac \pi 2 \frac 1 {p-1} \right) \ \left(\frac s p \right)^{\frac p {p-1}} \Big) \times \nonumber \\
& \times \cos \Big((p-1) \ \cos \left(\frac \pi 2 \frac 1 {p-1} \right) \ \left(\frac s p \right)^{\frac p {p-1}} - \frac \pi 4 \frac{p-2}{p-1} \Big) \ .
\end{align}
There is no explicit error estimate, but the basic order of decay is $\exp(-\alpha_p s^{\frac p {p-1}})$, and thus slower than $\exp(-d_p s^2)$ for $p > 2$. (In \cite{Bo} there is a misprint which is corrected in \eqref{eq3.12}.) For $p \notin 2 \N$ and $s << p^2$, \eqref{eq3.12} also seems to give a reasonably good approximation of $\Gamma(1+\frac 1 p) \gamma_p(s)$, as indicated by numerical plots. Actually, both orders for $p \in 2 \N$ and $p \notin 2 \N$ roughly coincide, i.e. $\exp(-\alpha_p s^{\frac p {p-1}}) \sim \beta_p s^{-(p+1)}$, if $s \sim  p^2 \ln(p)$. Thus a possible conjecture for $p \notin 2 \N$ would be that Proposition \ref{prop4} is valid for $s > c \ p^2 \ln(p)$, and that $\gamma_p$ has no real zeros for $s > c \ p^2 \ln(p)$ for a suitable constant $c>0$. P\'olya \cite{Po} showed that for $p \notin 2 \N$, $\gamma_p$ has only finitely many real zeros, but at least $2 [\frac p 2]$ of them, and infinitely many complex zeros.

\section{Proof of Theorem \ref{th3}}

In this section we prove the main technical result of the paper. Lemmas \ref{lem3} and \ref{lem4} show that the graph of $|\gamma_p|$ resembles the one of $\left| \frac{\sin(s)} s\right|$ for $s < p$. The estimate of the distribution function of $|\gamma_p|$ will be based on the following result on the distribution function of $|\frac{\sin(s)} s|$.

\begin{proposition}\label{prop5}
Let $F_{sinc}$ denote the distribution function of $|\frac{\sin(s)} s |$, i.e.
$$F_{sinc}(x) := \lambda\left( \{ s > 0 \ \Big| \ |\frac{\sin(s)} s| > x \} \right) \ , \ 0 < x < 1 \ . $$
Then for any $0 < x < \frac 1 {2 \pi}$
$$F_{sinc}(x) \ge \frac 2 \pi \frac 1 x - \frac{27}{16} \ . $$
\end{proposition}

\begin{proof}
Let $0 < x < \frac 1 {2 \pi}$, define $n \in \N$ and $0 \le \theta < 1$ by $\frac 1 {\pi x} = n + \theta$. Let $I_j := ( (j-1) \pi, j \pi]$ for $j= 1 \etc n$. For $s \in I_j$, $|\frac{\sin(s)} s| \ge |\frac{\sin(s)} {j \pi}|$ and hence with $j \pi x = \frac j {n + \theta}$
\begin{align*}
\lambda( \{ s \in I_j \ \big| \ |\frac{\sin(s)} s| > x \} ) &  \ge \lambda( \{ s \in I_j \ \big| \ |\sin(s)| > \frac j {n + \theta} \} ) \\
& = \pi - 2 \arcsin(\frac j {n+ \theta}) = 2 \arccos(\frac j {n + \theta})
\end{align*}
for $2 \le j \le n$ and $\lambda( \{ s \in I_1 \ | \ |\sin(s)| > \frac j {n + \theta} \} ) = \pi - \arcsin(\frac 1 {n + \theta}) = \arcsin(\frac 1 {n + \theta}) + 2 \arccos(\frac 1 {n+ \theta})$. Therefore
\begin{align*}
F_{sinc}(\frac 1 {n \pi}) & \ge \arcsin(\frac 1 {n + \theta}) + 2 \sum_{j=1}^n \arccos(\frac j {n + \theta}) \\
& = \arcsin(\frac 1 {n + \theta}) + 2 \sum_{j=0}^n \arccos(\frac j {n + \theta}) - \pi \ .
\end{align*}
Let $h(y) := \arccos(\frac y {n + \theta})$. By the Euler - Maclaurin summation formula
$$\sum_{j=0}^n h(j) = \int_0^n h(y) \ dy + \frac{h(0) + h(n)} 2 + \int_0^n (y - [y] - \frac 1 2) \ h'(y) \ dy \ . $$
We have $h'(y) = - \frac 1 {\sqrt{(n + \theta)^2 -y^2}} =: - k(y)$. Then $-\int_0^{\frac 1 2} (y - \frac 1 2) \  k(y) \ dy \ge 0$, and substituting $z = 2 j - y$ for $j = 1 \etc n$, we find
\begin{align*}
-\int_{j-\frac 1 2}^j & (y-(j-\frac 1 2)) \ k(y) \ dy - \int_j^{j+\frac 1 2} (y-(j+\frac 1 2)) \ k(y) \ dy \\
& = -\int_{j-\frac 1 2}^j (y-(j-\frac 1 2)) \ k(y) \ dy + \int_{j-\frac 1 2}^j (z-(j-\frac 1 2)) \ k(2 j - z) \ dz \ge 0 \ ,
\end{align*}
since $k$ is increasing in $y$. Therefore
$$-\int_0^n (y - [y] - \frac 1 2) \ k(y) \ dy \ge - \int_{n-\frac 1 2}^n (y - n + \frac 1 2) \ k(y) \ dy \ . $$
Further, substituting $y = (n + \theta) z$,
$$\int_0^n h(y) \ dy = (n + \theta) \int_0^{\frac n {n+\theta}} \arccos(y) \ dy = (n+ \theta) - (n+\theta) \int_{\frac n {n+\theta}}^1 \arccos(y) \ dy \ . $$
With $h(0) = \frac \pi 2$, $h(n) = \arccos(\frac n {n+ \theta})$ and $\arcsin(z) \ge z$ we find
\begin{align*}
F_{sinc}&(x)  \ge 2 (n+ \theta) - \frac \pi 2  + \frac 1 {n + \theta} - 2 (n+\theta) \int_{\frac n {n+\theta}}^1 \arccos(y) \ dy \\
& + \arccos(\frac n {n+ \theta}) - 2 \int_{n-\frac 1 2}^n \frac{y - n + \frac 1 2} {\sqrt{(n + \theta)^2 -y^2}} \  \ dy \\
& = 2 (n+ \theta) - \frac \pi 2  + \frac 1 {n + \theta} + 2 \arccos(\frac n {n+ \theta}) \\
& + (2n-1) \arccos(\frac{n-\frac 1 2}{n + \theta}) - 2 \sqrt{(\frac 1 2 + \theta)(2 n + \theta - \frac 1 2)} =: 2 (n+ \theta) - \frac \pi 2 + \phi_n(\theta) \ .
\end{align*}
For each $n \in \N$, the function $\phi_n$ is concave in $\theta$, since calculation shows
$$\phi_n''(\theta) = - \frac 1 {(n+\theta)^2} \left(\frac{2n (n^2+4n \theta + 2 \theta^2)}{(\theta (2n + \theta))^{\frac 3 2}} + \frac{(n-\frac 1 2)^2}{\sqrt{(\frac 1 2 + \theta)(2n + \theta - \frac 1 2)}} - \frac 1 {n+\theta}\right) < 0 \ . $$
Thus for any $n \in \N$, the minimum of $\phi_n$ is attained for either $\theta = 0$ or $\theta =1$. For $\theta =0$, we have
$\phi_n(0) = \frac 1 n + (2n-1) \arccos(1 - \frac 1 {2n}) - 2 \sqrt{n - \frac 1 4}$,
which has the derivative $\frac d {dn} \phi_n(0) = - \frac 1 {n^2} + 2 \arccos(1- \frac 1 {2n}) - 2 \frac{\sqrt{n - \frac 1 4}} n$. By the series development of $\arccos$, we have $2 \arccos(1-\frac 1 {2n}) \ge \frac 2 {\sqrt n} (1 + \frac 1 {24 n})$. Using this, one checks that $\frac d {dn} \phi_n(0) > 0$ for $n \ge 9$. Therefore $\phi_n(0) \ge \phi_9(0) > - 0.112$ for $n \ge 9$ and also for $n= 2 \etc 8$ by direct verification. A similar investigation shows that $\phi_n(1)$ attains larger (negative) values that $\phi_n(0)$. Since $-\frac \pi 2 - 0.112 > -\frac {27}{16}$, we conclude that
$F_{sinc}(x) \ge 2 (n + \theta) - \frac {27}{16} = \frac 2 \pi \frac 1 x - \frac {27}{16}$.
\end{proof}

\vspace{0,5cm}

For the investigation of the distribution function of $|\gamma_p|$ we need some information on the first $|\frac{\sin(s)} s|$-type bumps of the graph of $|\gamma_p|$ which we state in the next two Lemmas.

\begin{lemma}\label{lem6}
Let $p \ge 15$. Then $\gamma_p(s) > 0$ for all $0 \le s \le 3.11$ and $\gamma_p(s) < 0$ for $3.4 \le s \le 5.8$. If $\gamma_p(s) < - \frac 1 {100}$, then $s > \pi$. Further, $\gamma_p$ is strictly decreasing in $[\frac 1 4,4]$. For $s \in [0,\pi]$ we have: if $\gamma_p(s) < \frac 4 {17}$, $s \ge 2.48$, if $\gamma_p(s) < \frac 1 8$, $s \ge 2.75$, if $\gamma_p(s) < \frac 1 {10}$, $s > 2.8$ and if $\gamma_p(s) < \frac 1 {20}$, $s \ge 2.966$.
\end{lemma}

\begin{proof}
Let $p \ge 15$. Lemma \ref{lem4} yields with $\frac{14}{19} + \frac 5 {19} =1$ that
$$\gamma_p(s) \ge \Gamma(1+ \frac 1 p) \gamma_p(s) \ge \Phi_p(s) - \frac 1 {Np} \ge \frac{\min\left( \sin(s), \sin((1-\frac2 p)s) \right)} s \frac{\sin(\frac s p)}{\frac s p} - \frac 1 {Np} \ . $$
If $(1-\frac 2 p)s < \frac{\pi} 2$, we have $\frac{\sin((1-\frac 2 p)s)} s \ge (1-\frac 2 p) \frac 2 \pi \ge \frac {13}{15} \frac 2 \pi$, since $\frac{\sin(x)} x$ is decreasing in $[0, \pi]$, and $s \le \frac{15}{13} \frac\pi 2$ implies $\frac{\sin(s)} s \ge \frac{13}{15} \frac 2 \pi \sin(\frac{15}{13}\frac \pi 2)$. This yields $\Phi_p(s) \ge \frac{13}{15} \frac \pi 2 \sin(\frac{15}{13}\frac \pi 2) \frac{\sin(\frac \pi {26})}{\frac \pi {26}} > 0.53$, hence  $\gamma_p(s) \ge 0.53 - \frac 1 {120} > \frac 1 2$.
If $\frac{\pi} 2 \le (1-\frac 2 p) s \le s \le \pi$, $\sin((1-\frac 2 p)s) \ge \sin(s)$ and
$$\gamma_p(s) \ge \Phi_p(s) - \frac 1 {Np} \ge \frac{\sin(s)} s \frac{\sin(\frac s {15})}{\frac s {15}} - \frac 1 {120} =: h(s) \ . $$
The function $h$ is decreasing with $\gamma_p(2.48) \ge h(2.48) > \frac 4 {17}$, $\gamma_p(2.75) \ge h(2.75) > \frac 1 8$, $\gamma_p(2.8) \ge h(2.8) > \frac 1 {10}$, $\gamma_p(2.966) \ge h(2.966) > \frac 1 {20}$ and $\gamma_p(3.11) \ge h(3.11) > \frac 1 {600} > 0$. \\

By Lemma \ref{lem3}, $\Gamma(1+\frac 1 p) \gamma_p(s) \le  \frac{\sin(s)} s + \frac{1.016} p < \frac{\sin(s)} s + \frac 1 {14}$. Since $\frac{\sin(s)} s < - \frac 1 {14}$ for all $s \in [3.4,5.8]$, we have $\gamma_p(s) < 0$ there. Further, Lemma \ref{lem5} implies for $s \in [\frac 1 4, 4]$
$$\Gamma(1+\frac 1 p) \gamma_p'(s) \le \frac{s \cos(s) - \sin(s)}{s^2} + \frac 1 {14} < - \frac 1 {100} < 0 \ , $$
as a standard analysis of $k(s) = \frac{s \cos(s) - \sin(s)}{s^2}$ shows. Therefore $\gamma_p$ is strictly decreasing in $[\frac 1 4,4]$. (It is also decreasing in $[0,\frac 1 4]$ which we do not require here.) For $p \ge 15$, $\Gamma(1+\frac 1 p) \ge \frac{25}{26}$. If $\gamma_p(s) < - \frac 1 {100}$, Lemma \ref{lem4} yields that $\Phi_p(s) \le \Gamma(1+\frac 1 p) \gamma_p(s) + \frac 1 {120} < - \frac {25}{26} \frac 1 {100} + \frac 1 {120} < 0$. Since for $s \in[3,2 \pi]$ we have $\sin(\frac s p) > 0$ and $\alpha_p(s) >0$, we conclude from $\Phi_p(s) < 0$ that $\sin(s - \alpha_p(s)) < 0$ and hence $s > s - \alpha_p(s) > \pi$.
\end{proof}

\begin{lemma}\label{lem7}
Let $p \ge 15$, $x_1(p) := \max_{s \ge 3} |\gamma_p(s)|$, $x_2(p) := \max_{2 \pi \le s \le 3 \pi} |\gamma_p(s)|$. Then:
(a) $x_1(p) \in [0.1973,0.2336]$ and $x_2(p) \in [0.1011,0.1416]$. If $p \ge p_0 \simeq 26.265$, $x_1(p) \in [0.2010,0.2267]$ and $x_2(p) \in [0.1113,0.1360]$. \\
(b) Let $0 < x < x_1(p)$. Then $\gamma_p(s) = x$ has exactly one solution $0 < s_1(p) < 3.4$ and $\gamma_p(s) > x$ for all $0 < s < s_1(p)$. The equation $\gamma_p(s) = -x$ has exactly two solutions in $[3.11, 7]$, $3.11 < s_2(p) < s_3(p) < 7$ such that $|\gamma_p(s)| > x$ for all $s_2(p) < s < s_3(p)$.
If $\frac 1 8 \le x$, $s_2(p) \ge 3.55$, if $\frac 1 {10} \le x \le \frac 1 8$, $s_2(p) \ge 3.45$ and $s_3(p) \ge 5.39$ and if
$\frac 1 {20} \le s \le \frac 1{10}$, $s_2(p) \ge 3.27$ and $s_3(p) \ge 5.57$. If $0 < x \le \frac 1 {20}$, $s_3(p) \ge 5.89$ and $s_3(p)-s_2(p) \ge 2.426$. \\
Further, for $\frac 1 {100} < x \le \frac 1 {20}$, $\gamma_p(s) = x$ has exactly two solutions $s_4(p) < s_5(p)$ in $[5.5, 10]$ such that $|\gamma_p(s)| >x$ for all $s_4(p) < s < s_5(p)$. We have $s_5(p)-s_4(p) \ge 1.763$.
\end{lemma}

\begin{proof}
(a) By Corollary \ref{cor2}
\begin{align*}
\Gamma(1+ & \frac 1 p) |\gamma_p(s)|  \ge \Psi_p(s) \\
& = \frac 1 {19} \sqrt{221+140 \ \cos(2 \frac s p)} \ \frac{\min( |\sin(s-\alpha_p(s))|, |\sin(s)| )} s \ \frac{\sin(\frac s p)}{\frac s p} - \frac 1 {N p} \ ,
\end{align*}
and $\Psi_p(s)$ is increasing in $p$ for all $s \in [3.255,2 \pi]\cup [6.501,3 \pi]$. Hence $y_1(p) := \max_{s \in[3.255,2 \pi]} \Psi_p(s)$ is increasing in $p$, too. For $t_1 = 4.63$, $\Psi_{15}(t_1) \ge 0.19056$ and for $\tilde{t_1} = 4.58$, $\Psi_{p_0}(\tilde{t_1}) \ge 0.2010$. Therefore for any $15 \le p \le p_0$,
$\max_{s \ge 3} |\gamma_p(s)| \ge \frac {0.19056} {\Gamma(1+\frac 1 {p_0})} \ge 0.1973$ and for $p \ge p_0$, $\max_{s \ge 3} |\gamma_p(s)| \ge 0.2010$. The upper estimates $0.2336$ for $15 \le p \le p_0$ and $0.2267$ for $p \ge p_0$ were shown in Corollary \ref{cor3}, at least for $s \in [3,2 \pi]$. Note that for $s > 2 \pi$, $|\gamma_p(s)| < \frac 1 5 < 0.2267$, which is implied by Lemma \ref{lem4} and $\Gamma(1+\frac 1 p) |\gamma_p(s)| < \frac 1 s + \frac 1 {Np}$. \\
Similarly for $s \in I_2$, we choose $t_2 = 7.94$ with $\Psi_{15}(t_2) \ge 0.09767$ with $x_2(p) \ge \frac{0.09767}{\Gamma(1+\frac 1 {15})} \ge 0.1011$ and $\tilde{t_2} = 7.87$ with $\Psi_{p_0}(\tilde{t_2}) \ge 0.1113$. The upper estimates again follow from Corollary \ref{cor3}. \\

(b) Let $0 < x < x_1(p)$. By (a) $x_1(p)< 0.2336 < \frac 4 {17}$, and Lemma \ref{lem6} shows that $\gamma_p(s) = x$ requires $s > 2.48$. By Lemma \ref{lem6}, too, $\gamma_p(s)$ is strictly decreasing for $s$ in $[\frac 1 4 , 4]$ with $\gamma_p(3.4) < 0$.
Therefore $\gamma_p(s) = x$ has a unique solution $0 < s_1(p) < 3.4$ and $\gamma_p(s) > x $ holds for all $0 \le s < s_1(p)$. Lemma \ref{lem5} (b) implies that $\gamma_p(s) = -x$ has exactly two solutions $s_2(p) < s_3(p)$ in $[3.11,7]$, since $\gamma_p(3.11)>0$, $\gamma_p(5.5)<0$, $\gamma_p(7)>0$ and $\gamma_p$ is strictly convex in $[3.11,5.5]$ and strictly increasing in $[5.5,7]$. For $s_2(p) < s < s_3(p)$ we have $|\gamma_p(s)| > x$. By Corollary \ref{cor2} for $p \ge 15$ we have $|\gamma_p(s)| \ge \Psi_p(s) \ge \Psi_{15}(s)$ for all $s \in [3.255,2 \pi]$. Assume that $0 < x \le \frac 1 8$. Since $\Psi_{15}(5.38) \ge 0.127 > \frac 1 8$ and $\Psi_{15}$ is decreasing in $[5,6]$ because $|\sin(s)|$ is decreasing there, the larger solution satisfies $s_3(p) > 5.38$. Similarly, for $0 < x \le \frac 1 {10}$, with $\Psi_{15}(5.57) > \frac 1 {10}$, we have $s_3(p) > 5.57$. If $0 < x \le \frac 1 {20}$, we use $\Psi(5.89) > \frac 1 {20}$, and then $s_3(p) > 5.89$. \\
Concerning upper estimates for $|\gamma_p(s)|$ in $[3.255,4.4]$, where both $|\sin(s)|$ and $|\sin(s-\alpha_p(s))|$ are increasing, we find from Lemma \ref{lem4} for
$s \in [3.255,4.4]$
$$|\gamma_p(s)| \le \frac 1 {\Gamma(1+\frac 1 p)} \left( \left| \frac{\sin(s)} s \right| + \frac 1 {Np} \right) \le \frac 1 {\Gamma(1+\frac 1 {15})} \left( \left| \frac{\sin(s)} s \right| + \frac 1 {120} \right) =: h(s)  \ . $$
If $x \ge \frac 1 8$ and $\gamma_p(s) = -x$, $\frac 1 8 \le |\gamma_p(s)| \le h(s)$. Since $h(3.55) < \frac 1 8$ and $h$ is increasing in $[\pi,4.4]$, the smaller solution $s_2(p)$ satisfies $s_2(p) > 3.55$. If $x \ge \frac 1 {10}$, we use that $h(3.45) < \frac 1 {10}$, so that $s_2(p) > 3.45$. If $x \ge \frac 1 {20}$, use $h(3.27) < \frac 1 {20}$, hence $s_2(p) > 3.27$. For $x = \frac 1 {20}$, the equation $\Psi_{15}(s) = x$ has the two solutions $\bar{s_2} \simeq 3.472$ and $\bar{s_3} \simeq 5.898$ with $\bar{s_3}-\bar{s_2} > 2.426$. The two solutions $s_2(p) < s_3(p)$ of $\gamma_p(x) = -x$ are farther apart than $\bar{s_2}$ and $\bar{s_3}$, since $x = |\gamma_p(\bar{s_2})| \ge \Psi_p(\bar{s_2}) \ge \Psi_{15}(\bar{s_2})$ and $|\gamma_p(s)|$ is increasing near $\bar{s_2}$, so that $\bar{s_2} > s_2(p)$. Similarly $x = |\gamma_p(\bar{s_3})| \ge \Psi_{15}(\bar{s_3})$, where $|\gamma_p(s)|$ is decreasing near $\bar{s_3}$, so that $\bar{s_3} < s_3(p)$ and
$s_3(p) - s_2(p) \ge \bar{s_3} - \bar{s_2} > 2.426$. If $x < \frac 1 {20}$, the solutions of $\gamma_p(s) = -x$ are even farther apart than $2.426$. \\
On the interval $[5.5,10]$ the argument is similar: By Lemma \ref{lem5} $\gamma_p$ is strictly increasing in $[5.5,7]$ with $\gamma_p(5.5)<0<\gamma_p(7)$, strictly concave in $[7,8.5]$ and strictly decreasing in $[8.5,10]$ with $\gamma_p(10)<0$. By part (a), $\max_{s \in [5.5,10]} |\gamma_p(s)| > \frac 1 {10}$, so that for $0 < x \le \frac 1 {10}$, $\gamma_p(s) = x$ has exactly two solutions $s_4(p) < s_5(p)$ in $[5.5,10]$. For $s_4(p) < s < s_5(p)$ we have
$|\gamma_p(s)| > x$. For $x = \frac 1 {20}$, $\Psi_{15}(s) =x$ has the two solutions $\bar{s_4} \simeq 6.994$ and $\bar{s_5} \simeq 8.757$. Similarly as above,
$s_5(p) - s_4(p) \ge \bar{s_5} - \bar{s_4} \ge 1.763$.
\end{proof}

\vspace{0,5cm}

{\bf Proof of Theorem \ref{th3}} \\

(i) We will apply Proposition \ref{prop2}, the distribution function result of Nazarov, Podkorytov \cite{NP} to $\Omega = [0,\infty)$, $\mu= \lambda$ Lebesgue measure. \\
We start with case a), when $d_p \le c_p$, i.e. $26.265 \simeq p_0 \le p < \infty$ and put $f(s) = |\gamma_p(s)|$ and $g(s) = \exp(-d_p s^2)$. Then
$\int_0^\infty f(s)^2 \ ds = \int_0^\infty g(s)^2 \ ds = \frac 1 {2 \sqrt 2} \sqrt{\frac \pi {d_p}}$, see \eqref{eq2.5}. In terms of the distribution functions $F$ and $G$ of $f$ and $g$, we have $0 = \int_0^\infty (f(s)^2 - g(s)^2) \ ds = \int_0^1 2 x \ (F(x)-G(x)) \ dx$, which implies that $F=G$ has {\it at least} one solution in $(0,1)$. We have to show that $F-G$ changes sign {\it only once} in $(0,1)$. By Lemma \ref{lem6} $f$ is positive for all $0 \le s \le 3.11$ and has a zero in $[3.11,3.4]$ (which is actually $ > \pi$). By Proposition \ref{prop3} $f(s) = |\gamma_p(s)| = \gamma_p(s) \le \exp(-c_p s^2) \le \exp(-d_p s^2) = g(s)$ for all $0 \le s \le 3$, with strict inequality for $s>0$. Let again $x_1(p) := \max_{s \ge 3} f(s)$. By Lemma \ref{lem7}, $f(s) = |\gamma_p(s)| \le x_1(p) \le 0.2336$ for all $s \ge 3$ and $x_1(p) \ge 0.1973$. We claim that $f(s) < g(s)$ holds also for all $3 < s \le 3.3$ with $g(3.3) < x_1(p)$. By Lemma \ref{lem2} (a), $0.159 < \frac1 {2 \pi} < d_p \le d_{15} < 0.163$ for all $15 \le p < \infty$, hence
\begin{align*}
0.16 & < \exp(-0.163 \cdot 3.3^2) < \exp(-d_p \ 3.3^2) = g(3.3) \\
&\le \exp(-0.159 \cdot 3.3^2) < 0.18 < x_1(p) \ .
\end{align*}
The function $\Phi_p$ in Lemma \ref{lem4} satisfies
$$|\Phi_p(s)| \le \max \left( \frac{|\sin(s- \alpha_p(s))|} s , \frac{|\sin(s)|} s \right) \ , $$
and for $s \in [3,3.3]$, we have by Remark (1) following Lemma \ref{lem4}, $0 \le \alpha_p(s) \le \frac 5 {19} \frac 2 p 3.3 \le \frac{6.6}{57} < 0.12$, hence
$\max_{s \in [3,3.3]} |\Phi_p(s)| \le \max_{s \in [2.88,3.3]} \frac{|\sin(s)|} s < \frac 1 {11}$. Lemma \ref{lem4} yields for all $s \in [3,3.3]$ with $N$ as given there
\begin{align*}
f(s) & \le \max_{s \in [3,3.3]} |\gamma_p(s)| \le \frac 1 {\Gamma(1 + \frac 1 p)} (\frac 1 {11} + \frac 1 {Np}) \le \frac 1 {\Gamma(1 + \frac 1 {15})} (\frac 1 {11} + \frac 1 {120}) \\
& < 0.11 < 0.16 < \min_{s \in [3,3.3]} \exp(-d_p \ s^2) \le g(s) \ .
\end{align*}
This implies $F(x) < G(x)$ for all $x \in [x_1(p),1)$. \\

To prove that $F-G$ changes sign {\it at most once}, it suffices to show for a suitable $0 < x_2 < x_1(p)$ that
\begin{align}\label{eq4.1}
& (1) \; G-F \text{   is strictly increasing in   } [x_2,x_1(p)] \text{  and  } \nonumber \\
& (2) \; F(x) > G(x) \text{   for all   } x \in (0,x_2] \ .
\end{align}
Since $F$ and $G$ are differentiable decreasing functions, condition (1) is equivalent to $\frac{|F'(x)|}{|G'(x)|} > 1$, $x \in [x_2,x_1(p)]$. This means, see \cite{NP},
\begin{align}\label{eq4.2}
\frac{|F'(x)|}{|G'(x)|} & = \left(\sum_{s>0, |\gamma_p(s)| = x} \frac 1 {|\gamma_p'(s)|} \right) \ \frac 1 {|G'(x)|} \nonumber \\
& = 2 \sqrt{d_p} \sum_{s>0, |\gamma_p(s)| = x} \frac {|\gamma_p(s)|} {|\gamma_p'(s)|} \ \sqrt{\ln \left( \frac 1 {|\gamma_p(s)|} \right)} > 1 \ ,
\end{align}
where we used that $G(x) = \sqrt{\frac 1 {d_p} \ln ( \frac 1 x )}$, $\frac 1 {|G'(x)|} = 2 \sqrt{d_p} \ x \sqrt{\ln( \frac 1 x )}$. Having shown \eqref{eq4.1} and \eqref{eq4.2}, Proposition \ref{prop2} will imply that for all $u \ge 2$
$$\sqrt u \int_0^\infty |\gamma_p(s)|^u \ ds \le \sqrt u \int_0^\infty \exp(-d_p s^2)^u \ ds = \frac 1 2 \sqrt{\frac \pi {d_p}} \ , $$
the claim in Theorem \ref{th3} a). We will later choose $x_2 = \frac 1 {20}$ for all $p_0 \le p < \infty$ and also for $20 \le p \le p_0$. \\

Before proving \eqref{eq4.2} and \eqref{eq4.1}, let us consider case b), $c_p < d_p$, i.e. $20 < p < p_0$. Let $r :=r(p) := \frac{d_p}{c_p}$. It follows from Lemma \ref{lem2} that $1 < r < 1.0273$. Consider $f(s) = |\gamma_p(s)|^r$ and $g(s) = \exp(-d_p s^2)$ in Proposition \ref{prop2}. Then by Proposition \ref{prop3} and Lemma \ref{lem6} we have $f(s) = |\gamma_p(s)|^r \le \exp(-c_p s^2)^r = \exp(-d_p s^2)$ for $0 \le  s \le 3$, with strict inequality for $s>0$, $p \ge 15$. Again by Lemma \ref{lem7}, $x_1(p) := \max_{s \ge 3} |\gamma_p(s)|^r \ge 0.1973^{1.0273} \ge 0.1887$ with $g(3.3) \le 0.18 < x_1(p)$. Using the above estimate for $|\gamma_p(s)|$, we get
$f(s) = |\gamma_p(s)|^r \le |\gamma_p(s)| < g(s)$ for all $s \in [3,3.3]$. This implies $F(x) < G(x)$ for all $x_1(p) \le x < 1$.
Further, using $r>1$ and $\exp(-d_p s^2) \le 1$,
$$\int_0^\infty f(s)^{\frac 2 r} \ ds = \int_0^\infty \gamma_p(s)^2 \ ds = \int_0^\infty \exp(-d_p s^2)^2 \ ds < \int_0^\infty \exp(-d_p s^2)^{\frac 2 r} \ ds \ . $$
Suppose that $p \notin 2 \N$. By Proposition \ref{prop4}, $f(s) \sim \delta_p s^{-r(p+1)}$ for large $s$, so that $f \notin L_{\frac 1 {r(p+1)}}(0,\infty)$. This implies that there is $q$, $\frac 1 {r(p+1)} < q < \frac 2 r$ such that $\int_0^\infty f(s)^q \ ds = \int_0^\infty g(s)^q \ ds$, hence
$0 = \int_0^\infty (f(s)^q - g(s)^q) \ ds = q \int_0^1 x^{q-1} (F(x)-G(x)) \ dx$. Therefore $F=G$ has at least one solution in $(0,1)$. This is also true for $p \in 2 \N$, since then {\it on average} $f$ is larger than $g$ asymptotically, using \eqref{eq3.12}, hence for small $x$, $F(x) > G(x)$ will hold.
If $\tilde{F}$ and $\tilde{G}$ are the distribution functions of $|\gamma_p(s)|$ and $g(s)^{\frac 1 r} = \exp(-c_p s^2)$, we have $F(x) = \tilde{F}(x^{\frac 1 r})$ and $G(x) = \tilde{G}(x^{\frac 1 r})$. Thus it suffices to show that $\tilde{F} - \tilde{G}$ changes sign only once. Similar to \eqref{eq4.1} and \eqref{eq4.2} we will show for a suitable $0 < x_2 < x_1(p)$
\begin{align}\label{eq4.3}
& (1) \; \frac{|\tilde{F}'(x)|}{|\tilde{G}'(x)|} = 2 \sqrt{c_p} \sum_{s>0, |\gamma_p(s)| = x} \frac {|\gamma_p(s)|} {|\gamma_p'(s)|} \ \sqrt{\ln \left( \frac 1 {|\gamma_p(s)|} \right)} > 1 \ , x \in [x_2,x_1(p)] \ , \nonumber \\
& (2) \; \tilde{F}(x) > \tilde{G}(x) \text{   for all   } x \in (0,x_2] \ .
\end{align}
Having shown this, Proposition \ref{prop2} will imply that for all $v \ge u_0 = \frac 2 r$, $\int_0^\infty f^v(s) \ ds \le \int_0^\infty g(s)^v \ ds$, i.e. for all $u = r v \ge 2$
$$\sqrt u \int_0^\infty |\gamma_p(s)|^u \ ds \le \sqrt u \int_0^\infty \exp(-d_p s^2)^v \ ds = \sqrt u \int_0^\infty \exp(-c_p s^2)^u \ ds = \frac 1 2 \sqrt{\frac \pi {c_p}} \ , $$
the claim in Theorem \ref{th3} b). We will again choose $x_2 = \frac 1 {20}$. \\

(ii) We note that $c_p$ and $d_p$, which are important for the estimates in \eqref{eq4.1}, \eqref{eq4.2} and \eqref{eq4.3}, do not vary very much in their respective intervals. We have
$$\inf_{p_0 \le p < \infty} d_p = \frac 1 {2 \pi} \simeq 0.1592 < \max_{p_0 \le p < \infty} d_p \simeq 0.1610 $$
$$\text{   and   } \min_{15 \le p < p_0} c_p \simeq 0.15846 < \max_{15 \le p < p_0} c_p \simeq 0.1610 \ , $$
so that $\min( 2 \sqrt{c_p} , 2 \sqrt{d_p} ) \ge 0.796$ and $\max( \sqrt{\frac 1 {c_p}} , \sqrt{\frac 1 {d_p}} ) \le 2.5122$ in the ranges for $p$ considered here.\\
In both cases a) and b) we have to compare the distribution functions $F$ of $|\gamma_p|$ and $G$ of $\exp(-d_p s^2)$, case a) or $\exp(-c_p s^2)$, case b). It suffices to prove with the choice $x_2 = \frac 1 {20}$ \\
\begin{align*}
& (1) \; h_p(x) := 0.796 \sum_{s >0, |\gamma_p(s)| = x} \frac{|\gamma_p(s)|}{|\gamma_p'(s)|} \sqrt{\ln ( \frac 1 {|\gamma_p(s)|} )} > 1 \ , \ x \in [x_2,x_1(p)] \ , \\
& (2) \ F(x) > 2.5122 \ \sqrt{\ln( \frac 1 x )} \ge G(x) \ , x \in (0,x_2] \ .
\end{align*}
Then the unique sign change of $G-F$ will occur at some $x_0 \in (x_2,x_1(p))$. \\

\vspace{0,3cm}

We now verify (1) for all $p \ge 15$ and $x$ with $x_2 = \frac 1 {20} \le x \le x_1(p)$. We have $h_p(x) = 0.796 \ ( \sum_{s >0, |\gamma_p(s)| = x} \frac 1 {|\gamma_p'(s)|} ) \ x  \sqrt{\ln(\frac 1 x)}$. For $0 < x < \frac 3 5 < \exp(-\frac 1 2)$, $x  \sqrt{\ln(\frac 1 x)}$ is increasing in $x$. Suppose first that $\frac 1 8 \le x \le x_1(p)$. By Lemma \ref{lem5} for all $p \ge 15$
$$|\gamma_p'(s)| \le \frac{1.064}{\Gamma(1+\frac 1 p)} (\frac 1 s + \frac 1 p) \le \frac{1.064}{\Gamma(1+\frac 1 {15})} (\frac 1 s + \frac 1 {15}) \le 1.102 \ (\frac 1 s + \frac 1 {15}) =: l_1(s) \ ,  $$
for $0 < s < \pi$, while
$$|\gamma_p'(s)| \le \frac{1.016}{\Gamma(1+\frac 1 {15})} (\frac 1 s + \frac 1 {15}) \le 1.053 \ (\frac 1 s + \frac 1 {15}) =: l_2(s) $$
for $s \ge \pi$. By Lemmas \ref{lem6} and \ref{lem7}, with $x_1(p) < \frac 4 {17}$ as needed there, $|\gamma_p(s)| = x$ has three solutions in $(0,7)$, $s_1(p) \ge 2.48$, $s_3(p) > s_2(p) \ge 3.55$. Hence
$$h_p(x) \ge 0.796 \ ( \frac 1 {l_1(2.48)} + \frac 2 {l_2(3.55)} ) \frac 1 8 \sqrt{\ln 8} > 1.06 > 1 \ . $$
For $\frac 1 {10} \le x \le \frac 1 8$, the three solutions satisfy $s_1(p) \ge 2.75$, $s_2(p) \ge 3.45$ and $s_3(p) \ge 5.39$, see Lemmas \ref{lem6} and \ref{lem7}. Hence for these $x$
$$h_p(x) \ge 0.796 \ ( \frac 1 {l_1(2.75)} + \frac 1 {l_2(3.45)} + \frac 1 {l_2(5.39)} ) \frac 1 {10} \sqrt{\ln 10} > 1.03 > 1 \ . $$
For $\frac 1 {20} \le x \le \frac 1 {10}$, by Lemma \ref{lem7} (a), $x \le \frac 1 {10} < x_2(p)$, and $|\gamma_p(s)| = x$ has two additional solutions $s_5(p) > s_4(p)$ in $(5.5, 10)$. We have $s_1(p) \ge 2.966$, $s_2(p) \ge 3.27$, $s_3(p) \ge 5.57$, $s_5(p) > s_4(p) \ge s_3(p) \ge 5.57$. Therefore
$$h_p(x) \ge 0.796 \ ( \frac 1 {l_1(2.966)} + \frac 1 {l_2(3.27)} + \frac 3 {l_2(5.57)}) \frac 1 {20} \sqrt{\ln 20} > 1.1 > 1 \ . $$
Hence (1) is satisfied for all $p \ge 15$ and $x$ with $\frac 1 {20} \le x \le x_1(p)$. For $x = \frac 1 {20}$, we have by Lemmas \ref{lem6} and \ref{lem7}, that $|\gamma_p(s)| > x$ holds for all $0 \le x < 2.966$, $s_2(p) < x < s_3(p)$ with $s_3(p) - s_2(p) \ge 2.425$ and for $s_4(p) < x < s_5(p)$ with $s_5(p) - s_4(p) \ge 1.763$. Thus $F(\frac 1 {20}) \ge 7.15 > 4.25 > G(\frac 1 {20})$. Actually $G(x) \le 2.5122 \sqrt{\ln(\frac 1 x)} \le 7.15$ holds for all $\frac 1 {3200} < x < \frac 1 {20}$. Hence $F(x) > G(x)$ is true for all $x$ with $\frac 1 {1000 \pi} < x < \frac 1 {20}$. \\

(iii) Suppose that $p \notin 2 \N$. Let $A > 1$ and $p \ge 15$ be such that $|\sin(\frac \pi 2 p)| \ge \frac 1 {A^p}$, which is satisfied if dist$(p,2 \N) \ge \frac 1 {A^p}$. We prove that $F(x) > G(x)$ holds for {\it very small} $x>0$, namely for $0 < x \le \left(\frac{0.482}{(Ap)^2}\right)^{p+1}$. By Stirling's formula $\Gamma(p+1) \ge \sqrt{2 \pi p} (\frac p e)^p$ and hence
$$\left(\frac 1 2 \frac{\Gamma(p+1)}{\Gamma(1+\frac 1 p)} \right)^{\frac 1 {p+1}} \ge \left(\frac 1 2 \Gamma(p+1) \right)^{\frac 1 {p+1}} \ge \sqrt{\frac{\pi e^2}{2 p}}^{\frac 1 {p+1}} \frac p e \ . $$
The function $k(p) := \sqrt{\frac{\pi e^2}{2 p}}^{\frac 1 {p+1}}$ satisfies $(\ln k)'(p) = \frac 1 {p(p+1)^2} (p \ln (\frac{2p}{\pi e^2}) - 1)$. Since
$\frac 2 {\pi e^2} \simeq \frac 1 {31.5}$, $(\ln k)'(p_1) = 0$ for $p_1$ near $31.5$, namely $p_1 \simeq 32.535$, and $k$ attains its minimum at $p_1$, with $k(p) \ge k(p_1) > 0.9847$. This implies $\left(\frac 1 2 \Gamma(p+1) \right)^{\frac 1 {p+1}} > 0.3622 \ p$. By Proposition \ref{prop4} we have for any $s > \frac 5 8 A p^3$ that
$$|\gamma_p(s)| \ge \frac 1 2 \frac{\Gamma(p+1)}{\Gamma(1+\frac 1 p)} \frac 1 {A^p} \frac 1 {s^{p+1}} \ge \left(\frac{0.3622 \ p} s \right)^{p+1} \frac 1 {A^p} \ge \left(\frac{0.3622 \ p} {A s} \right)^{p+1} \ . $$
Thus for $x > 0$
\begin{align*}
F(x) & = \lambda \{ s > 0 \ \Big| \ |\gamma_p(s)| > x \} \ge \lambda \{ s > \frac 5 8 A p^3 \ \Big| \ \left(\frac{0.3622 \ p} {A s} \right)^{p+1} > x \} \\
& = \lambda \{ s > 0 \ \Big| \ \frac 5 8 A p^3 < s < \frac{0.3622 \ p} {A x^{\frac 1 {p+1}}} \ \} \ .
\end{align*}
Assuming $0 < x < \left(\frac{0.482 p}{(A p)^2}\right)^{p+1}$, $\frac{0.3622 \ p} {A x^{\frac 1 {p+1}}} > \frac 3 4 A p^3 = \frac 6 5 (\frac 5 8 A p^3)$ and hence
$$F(x) \ge \frac{0.3622 \ p} {A x^{\frac 1 {p+1}}} - \frac 5 8 A p^3 \ge \frac 1 6 \frac{0.3622 \ p} {A x^{\frac 1 {p+1}}} \ . $$
For $G$ we know $G(x) < 2.5122 \sqrt{\ln(\frac 1 x)}$. To prove $F(x) > G(x)$, it suffices to verify $x^{\frac 1 {p+1}} \sqrt{\ln(\frac 1 x)} \le \frac 1 {50} \frac p A$, since $\frac{0.3622}{6 \cdot 2.5122} > \frac 1 {50}$. Now $l(x) := x^{\frac 1 {p+1}} \sqrt{\ln(\frac 1 x)}$ is increasing for $0 < x < \exp(-\frac{p+1} 2)$, which is satisfied under our assumption on $x$. Inserting the maximal $x$ considered, \\ $\bar{x} = \left(\frac{0.482}{(Ap)^2}\right)^{p+1}$, we find
$l(\bar{x}) = \frac{0.482}{(A p)^2} \sqrt{(p+1) [ 2 \ln(A p) + \ln(\frac 1 {0.482}) ] }$, which is easily seen to be $< \frac 1 {50}\frac p A$ for any $A \ge 1$ and $p \ge 15$. \\
For the maximal $x$ in this range, with $\ln(\frac 1 {0.482}) \le 0.73$,
\begin{equation}\label{eq4.4}
G(\bar{x}) \le 2.5122 \sqrt{(p+1) [ 2 \ln(A p) + 0.73 ] }  \ .
\end{equation}

(iv) It remains to prove $F(x) > G(x)$ for all $x$ in the intermediate range $\left(\frac{0.482}{(Ap)^2}\right)^{p+1} < x < \frac 1 {1000 \pi}$ for $p \ge 20$.
Assume first that $\frac 1 {210 p} \le x \le \frac 1 {1000 \pi}$. By Lemma \ref{lem3} $|\gamma_p(s)| \ge \frac{|\sin(s)|} s - \frac {1.016} p $. Therefore by Proposition \ref{prop5}
$$F(x) \ge F_{sinc} \left(x + \frac {1.016} p \right) > \frac 2 \pi \frac 1 {x + \frac {1.016} p} - \frac{27}{16} \ .$$
Since $G(x) \le 2.5122 \sqrt{\ln(\frac 1 x)}$, $F(x) > G(x)$ will be satisfied if
$$\left(x+\frac {1.016} p \right) \left(\sqrt{\ln(\frac 1 x)} + \frac{27}{40} \right) < \frac 1 4 \ , $$
using that $\frac{27}{16 \cdot 2.5122} < \frac{27}{40}$ and $\frac 2 {2.5122 \ \pi} > \frac 1 4$. By assumption $\frac 1 p < 210 \ x$, hence $x + \frac{1.016} p < 215 \ x$. It suffices that $215 \ x (\sqrt{\ln(\frac 1 x)} + \frac{27}{40}) < \frac 1 4$ holds. Since $x \sqrt{\ln(\frac 1 x)}$ is increasing for $0 < x < \exp(-\frac 1 2)$, we only have to check this for the maximal $x=\frac 1 {1000 \pi}$, when it is true. \\
If $x < \frac 1 {210 p}$, $x + \frac{1.016} p < \frac{1.021} p$, and by the above estimate $F(x) \ge \frac 2 {1.021 \ \pi} p - \frac{27}{16}$. This will be $> 2.5122 \sqrt{\ln(\frac 1 x)} \ge G(x)$, if $\sqrt{\ln(\frac 1 x)} < 0.2481 p - \frac{27}{40}$. For any $p \ge 20$, $(0.2481 p - \frac{27}{40})^2 \ge \frac{p^2}{25}$. Thus $F(x) > G(x)$ holds for all $x > 0$ with $\exp(-\frac{p^2}{25}) < x < \frac 1 {210 p}$. \\

Assume next that $\bar{x} := \left( \frac{0.482}{(A p)^2} \right)^{p+1} < x < \exp(-\frac{p^2}{25})$. By \eqref{eq4.4}
$$G(x) \le G(\bar{x}) = 2.5122 \sqrt{(p+1) [ 2 \ln(A p) + 0.73 ] } \le 3.5528 \sqrt{(p+1) [\ln(A p) + 0.365]} \ . $$
For $x < \exp(-\frac{p^2}{25})$, $x + \frac{1.016} p \le \frac{1.018} p$ and the estimate for $F(x)$ yields $F(x) \ge \frac 2 {1.018 \pi} p - \frac{27}{16} \ge \frac 5 8 - \frac{27}{16}$. To verify $F(x) > G(x)$, it suffices to show
\begin{equation}\label{eq4.5}
\psi_A(p) := \frac 5 8 p - \frac{27}{16} - 3.5528 \sqrt{(p+1) [\ln(A p) + 0.365]} > 0 \ .
\end{equation}
For any $A \ge 1$, $\psi_A$ is increasing in $p$. For the choice $A=p$, $\psi_A(400) > 2 > 0$ and for $A=10$, $\psi_A(265) > 1 >0$. Together with parts (ii) and (iii) this shows that $F(x) > G(x)$ holds for all $0 < x < x_1(p)$ for $A=p$ if $p \ge 400$ and for $A=10$ if $p \ge 265$. This proves Theorem \ref{th3} for $p \ge 400$. \\

(v) For $20 \le p \le 400$, we use the better approximation of $|\gamma_p(s)|$ of Lemma \ref{lem4}. We may also assume that
$\bar{x} = \left( \frac{0.482}{(A p)^2} \right)^{p+1} < x < \exp(-\frac{p^2}{25})$, since the estimate outside of this interval was valid for all $p \ge 20$.
By Lemma \ref{lem4}, $|\gamma_p(s)| \ge |\Phi_p(s)| - \frac 1 {N p}$, where $N$ is as in Lemma \ref{lem3}, e.g. $N=8.002$ for $26 \le p \le 175$, and
$\Phi_p(s) = k(\frac s p) \frac{\sin(s- \alpha_p(s))} s$, $k(y) = \frac 1 {19} \sqrt{221+140 \cos(2 y)} \frac{\sin(y)} y$. Since for $x < \exp(-\frac{p^2}{25})$ and $26 \le p \le 175$ we have $x + \frac 1 {8.002 p} < \frac 1 {8p}$,
$$F(x) \ge \lambda \{ s > 0 \ | \ |\Phi_p(s)| > \frac 1 {M p} \} \ , \ M=8 \ , \ 26 \le p \le 175 \ , \  \bar{x} < x < \exp(-\frac{p^2}{25}) \ . $$
For $175 < p < \infty$, the same holds with $M=7.85$, and for $17 \le p \le 26$ with $M=8.6$. We will use this for $0 \le s \le \frac 3 4 \pi p$, $y= \frac s p \le \frac 3 4 \pi$. The function $k$ is decreasing in $y \in [0,\frac 3 4 \pi]$, since
$$(\ln k)'(y) = - \frac{140 \sin(2y)}{221+140 \cos(2y)} + \cot(y) - \frac 1 y \le - \frac{140 \sin(2y)}{221+140 \cos(2y)} - \frac y 3 - \frac{y^3}{45} \ , $$
using the series development of cotangent. For $0 < y \le \frac \pi 2$, this is obviously negative, and for $\frac \pi 2 < y < \frac 3 4 \pi$, too, as may be checked. Therefore for $s \in ((j-1) \pi, j \pi]$, $j \le \frac 3 4 p$, $k(\frac s p) \ge k(\frac {j \pi} p)$. \\

Let $J_j = ((j-1) \pi - \alpha_p((j-1) \pi),j \pi - \alpha_p(j \pi)]$ for $j = 1 \etc n$. Recall that $\alpha_p(0)=0$, that $\alpha_p(s)$ is increasing for $0 < s < (\pi - \arccos(\frac 5 {14})) \frac p 2 \simeq 0.968 p$ and decreasing for $(\pi - \arccos(\frac 5 {14}) ) \frac p 2< s <(\pi + \arccos(\frac 5 {14}) ) \frac p 2 \simeq 2.173 p$. Let $y_j = \frac 1 M \frac{\frac{j \pi} p}{k(\frac{j \pi} p)}$, essentially $y_j \simeq \frac 1 M \frac{j \pi} p$. Solving $|\sin(s - \alpha_p(s))| = y_j$ for $s$ in $J_j$ yields two solutions $s_{j,1} < s_{j,2}$ with
$s_{j,1} - \alpha_p(s_{j,1}) = \pm \arcsin(y_j)$, $s_{j,2} - \alpha_p(s_{j,2}) = \pm (\pi - \arcsin(y_j))$, hence $\pm 2 \arccos(y_j) = \pm (\pi - 2 \arcsin(y_j)) = s_{j,2} - s_{j,1} - (\alpha_p(s_{j,2}) - \alpha_p(s_{j,1}))$. In each interval $J_j$, the difference $\alpha_p(s_{j,2}) - \alpha_p(s_{j,1})$ is small, of modulus $< \frac 2 p$, positive for $0 < s_{j,i} < 0.96 p$ and negative for $s_{j,i} > 0.97 p$. The gaps between successive differences $\alpha_p(s_{j+1,1}) - \alpha_p(s_{j,2})$ are even smaller, of order $\frac 1 {p^2}$, since $0 < s_{j+1,1} - s_{j,2} = O(\frac 1 p)$ (for small $j$ about $\frac{2 \pi} M \frac 1 p < \frac 1 p$) and $|\alpha_p'(s)| \le \frac{10} 9 \frac 1 p$. The gaps slightly increase with $j$. Further, $\frac{|\alpha_p(s)|} s \le \frac 5 9 \frac 1 p$ for all $s>0$. The distribution function of $\Phi_p$ will be bounded below by sums of $s_{j,2} - s_{j,1}$. The error by sums of $-(\alpha_p(s_{j,2}) - \alpha_p(s_{j,1}))$ over $j$ is bounded in modulus by the maximum of $|\alpha_p(s)|$, which is $\arctan(\frac 5 {3 \sqrt{19}}) < 0.366$, attained at $\pi \pm \arccos(\frac 5 {14})$. We will estimate the distribution function of $|\Phi_p|$ from below by assuming $\alpha_p(s) =0$ and correct the possible error by subtracting from the lower bound the value $0.366$. \\

Hence consider $I_j = ((j-1) \pi, j \pi]$, $j \in \N$ and $|\Phi_p(s)| = k(\frac s p) \frac{|\sin(s)|} s$. For $s \in I_j$, $|\Phi_p(s)| \ge k(\frac {j \pi} p) \frac{|\sin(s)|} {j \pi}$ and
$$\lambda \{ s \in I_j \ \Big| \ |\Phi_p(s)| > \frac 1 {M p} \} \ge \lambda \{ s \in I_j \ \Big| \ |\sin(s)| > \frac 1 M \frac{\frac{j \pi} p}{k(\frac {j \pi} p)} \} \ . $$
Obviously, we require for $y=\frac{j \pi} p$ that $l(y):= \frac 1 M \frac y {k(y)} = \frac{19} M \frac{y^2}{\sin(y)} \frac 1 {\sqrt{221+140 \cos(2y)}} \le 1$. Since $k$ is decreasing, $l$ is increasing and $l(y)=1$ has a unique solution $y_0 > 0$. For $M = 8$ it is $y_0 \simeq 2.1228$, for $M=7.85$, $y_0 \simeq 2.0976$ and for $M=8.6$, $y_0 \simeq 2.2096$. As required above, $y_0 \le \frac 3 4 \pi$. Define $n \in \N$ and $\theta \in [0,1)$ by $\frac {y_0} \pi p = n + \theta$. Then $\frac{ j \pi} p = y_0 \frac j {n + \theta} \le y_0$ for all $j = 1 \etc n$ and $l(\frac{j \pi} p) \le 1$. We get for $j = 2 \etc n$
\begin{align*}
\lambda \{ s \in I_j \ \Big| \ |\sin(s)| > \frac 1 M \frac{\frac{j \pi} p}{k(\frac {j \pi} p)} \} & = \lambda \{ s \in I_j \ \Big| \ |\sin(s)| > l \left(\frac {y_0 j} {n + \theta} \right) \} \\
& = 2 \arccos \left( l \left(\frac {y_0 j} {n+ \theta} \right) \right) \ ,
\end{align*}
with an additional term $\arcsin \left(l \left(\frac {y_0} {n + \theta} \right) \right)$ for $j=1$ and
\begin{align}\label{eq4.6}
F(x) & \ge  2 \sum_{j=1}^n \arccos \left(l \left(\frac {y_0 j} {n+ \theta} \right) \right) + \arcsin \left(l \left(\frac {y_0} {n + \theta} \right) \right) \nonumber \\
& \ge  2 \sum_{j=0}^n \arccos \left(l \left(\frac {y_0 j} {n+ \theta} \right) \right) - \pi + l \left(\frac {y_0} {n + \theta} \right)  \ ,
\end{align}
also using $\arcsin(x) \ge x$ for $0 \le x \le 1$. Let $h(u) := \arccos \left(l \left(\frac {y_0 u} {n + \theta} \right) \right)$.
By the Euler - Maclaurin summation formula and the substitution $v = (n + \theta) u$,
\begin{align}\label{eq4.7}
\sum_{j=0}^n & h(j) = \int_0^n h(u) \ du + \frac{h(0) + h(n)} 2 + \int_0^n (u - [u] - \frac 1 2) \ h'(u) \ du \nonumber \\
& = (n + \theta) \int_0^{\frac n {n+\theta}} \arccos(l(y_0 v)) \ dv + \frac 1 2 \left[ \frac \pi 2 + h(n) \right] + \int_0^n (u - [u] -\frac 1 2) h'(u) \ du \ .
\end{align}
Similar as in the proof of Proposition \ref{prop5}, we have $h'(u) < 0$ for all $0 < u < 1$ and
$$\int_0^n (u - [u] -\frac 1 2) h'(u) \ du \ge \int_{n - \frac 1 2}^n (u - [u] -\frac 1 2) h'(u) \ du  \ . $$
In the proof of Proposition \ref{prop5} we used that $|h'(u)|$ was increasing, which is true here for $0 < u < 1, u \notin [0.7,0.85]$. However, while this may be false when $u \in [0.7,0.85]$, $|h'(u-c) +h'(u)|$ is then increasing, when $c \in [\frac 1 5, \frac 1 4]$, as a somewhat tedious investigation shows. If $u = \frac{y_0 j}{n + \theta}$ is in $[0.7,0.85]$, choose $k < j$ such that $v := \frac{y_0 k}{n + \theta}$ satisfies $\frac 1 5 < u-v < \frac 1 4$ to conclude for $K_j = [j-\frac 1 2, j + \frac 1 2]$ and $K_k = [k-\frac 1 2, k + \frac 1 2]$ that $\int_{K_j \cap K_k} ((u - [u] -\frac 1 2) h'(u) \ du \ge 0$. Combining any two such intervals, if necessary, we get, since $h'(u) < 0$,
$$\int_0^n (u - [u] -\frac 1 2) h'(u) \ du \ge \frac 1 2 \int_{n-\frac 1 2}^n h'(u) \ du = \frac 1 2 \left(h(n) - h(n-\frac 1 2)\right) \ . $$
Further, $\int_0^{\frac n {n+\theta}} \arccos(l(y_0 v)) \ dv = \int_0^1 \arccos(l(y_0 v)) \ dv - \int_{\frac n {n+\theta}}^1 \arccos(l(y_0 v)) \ dv$ and
$(n + \theta) \int_{\frac n {n+\theta}}^1 \arccos(l(y_0 v)) \ dv \le (n+ \theta) (1-\frac n {n+\theta}) \arccos(l(\frac{y_0 n}{n + \theta})) = \theta h(n)$, since $h$ is decreasing. Inserting this into equation \eqref{eq4.7}, we find with $n + \theta = \frac{y_0} \pi p$
\begin{align*}
\sum_{j=0}^n h(j) & \ge \frac{y_0} \pi p \int_0^1 \arccos(l(y_0 v)) \ dv + (1-\theta) h(n) - \frac 1 2 h(n-\frac 1 2) + \frac \pi 4 \\
& \ge \frac{y_0} \pi p \int_0^1 \arccos(l(y_0 v)) \ dv - \frac 1 2 h(n-\frac 1 2) + \frac \pi 4 \ .
\end{align*}
Together with inequality \eqref{eq4.6}, we get
$$F(x) \ge 2 \frac{y_0} \pi p \int_0^1 \arccos(l(y_0 v)) \ dv - \frac \pi 2 - \arccos\left(l\left(\frac{y_0 (n-\frac 1 2)}{n + \theta}\right)\right) + l\left(\frac {y_0} {n + \theta}\right) \ . $$
As mentioned before, to take into account the effect of the omitted $\alpha_p$-terms, we subtract from this the value
$\arctan(\frac 5 {3 \sqrt{19}})$. Since $\frac \pi 2 + \arctan(\frac 5 {3 \sqrt{19}}) < 1.9361$, we arrive at
\begin{equation}\label{eq4.8}
F(x) \ge 2 \frac{y_0} \pi p \int_0^1 \arccos(l(y_0 v)) \ dv - 1.9361 - \arccos\left(l\left(\frac{y_0 (n-\frac 1 2)}{n + \theta}\right)\right) + l\left(\frac {y_0} {n + \theta}\right) \ .
\end{equation}
Consider first the case $175 \le p \le \infty$. Then $y_0 \ge 2.0976$ and numerical integration yields $L:= \int_0^1 \arccos(l(y_0 v)) \ dv \ge 1.15206$. We have $n + \theta = \frac{y_0} \pi p \ge 116$ and then $- \arccos\left(l\left(\frac{y_0 (n-\frac 1 2)}{n + 1}\right)\right) + l\left(\frac {y_0} {n + 1}\right) \ge -0.1935$ so that with $2 \frac{y_0} \pi L \ge 1.5384$, $F(x) \ge 1.5384 p - 2.13$. In the case $26 \le p \le 175$ we have similarly with $y_0 \ge 2.1228$ by numerical integration that $L = \int_0^1 \arccos(l(y_0 v)) \ dv \ge 1.15197$. Then $n + \theta = \frac{y_0} \pi p \ge 17.5$ implies that
$- \arccos\left(l\left(\frac{y_0 (n-\frac 1 2)}{n + 1}\right)\right) + l\left(\frac {y_0} {n + 1}\right) \ge -0.4639$, with the conclusion $F(x) \ge 1.5568 p - 2.40$. Finally, in the case $20 \le p \le 26$, $y_0 \ge 2.2096$, $L \ge 1.15628$ and $n + \theta \ge 14.05$, with
$- \arccos\left(l\left(\frac{y_0 (n-\frac 1 2)}{n + 1}\right)\right) + l\left(\frac {y_0} {n + 1}\right) \ge -0.5414$, so that $F(x) \ge 1.6265 p - 2.478$.
We note that the integral $L = \int_0^1 \arccos(l(y_0 v)) \ dv$ does not depend very much on $M$, being with $\simeq 1.15$ not much bigger that
$\int_0^1 \arccos(v) \ dv = 1$, which was used in the previous estimate \eqref{eq4.5}. (Actually, $l(y_0 v) \le v$ and $\arccos(l(y_0 v)) \ge \arccos(v)$ for all $0 \le v \le 1$.) The factor of $p$ in \eqref{eq4.5} was only $\frac 5 8$, while it is now $> \frac 3 2$, as a consequence of the better approximation of $\gamma_p$ in Lemma \ref{lem4}. Corresponding to \eqref{eq4.5} we have to check that
$F(x) -  3.5528 \sqrt{(p+1) [\ln(A p) + 0.365]} > 0$. In the first case $175 < p < \infty$, this will be satisfied, if
$$\psi_A(p) := 1.5384 p - 2.13 -  3.5528 \sqrt{(p+1) [\ln(A p) + 0.365]} > 0 \ . $$
Since $\psi$ is increasing in $p$, we only have to verify this for a minimal $p$, given a fixed $A>1$. For $A=p$ it is satisfied for $p \ge 50$, $\psi_A(50)>2$, and hence automatically for all $p \ge 175$. For $p \ge 175$, we may even choose $A=p^5$. In the range $26 \le p \le 175$, the condition is
$$\psi_A(p) := 1.5568 p - 2.40 -  3.5528 \sqrt{(p+1) [\ln(A p) + 0.365]} > 0 \ , $$
which is satisfied for $A=2$ and $p \ge 26.5$. For $26.265 \simeq p_0 \le p \le 26.5$, $dist(p,2 \N) > 2^{-p}$ is automatically satisfied and $A = \frac 3 2$ suffices for $\psi_A(p_0) > 0$. For $A=10$, it is satisfied for $p \ge 37$ and for $A=p$ for $p \ge 46$. Consider now $20 \le p \le p_0$. We then have to check
$$\psi_A(p) := 1.6265 p - 2.478 -  3.5528 \sqrt{(p+1) [\ln(A p) + 0.365]} > 0 \ . $$
This is true, if $p \ge 20.2$ and $A=15/14$; for $20 \le p \le 20.2$, $dist(p, 2 \N) < A^{-p}$ holds, which is excluded in Theorem \ref{th3}. This proves Theorem \ref{th3} also for $20 < p \le 400$.    \hfill $\Box$  \\

{\bf Remark.} To prove Theorem \ref{th3} (b) also for $2 < p < 20$, a reasonably good approximation of $\gamma_p(s)$ by damped trigonometric functions for $p \le s \le p^2$ would be helpful. \\

\vspace{0,5cm}

\vspace*{0.5cm}

\noindent Mathematisches Seminar \\
Universit\"at Kiel \\
24098 Kiel, Germany \\
hkoenig@math.uni-kiel.de \\

\end{document}